\documentclass{amsart}
\usepackage{amssymb}
\usepackage{amscd}
\usepackage{verbatim}
\usepackage{epsfig}

\begin{document}
\newcommand\Mand{\ \text{and}\ }
\newcommand\Mor{\ \text{or}\ }
\newcommand\Mfor{\ \text{for}\ }
\newcommand\Real{\mathbb{R}}
\newcommand\RR{\mathbb{R}}
\newcommand\im{\operatorname{Im}}
\newcommand\re{\operatorname{Re}}
\newcommand\sign{\operatorname{sign}}
\newcommand\sphere{\mathbb{S}}
\newcommand\BB{\mathbb{B}}
\newcommand\HH{\mathbb{H}}
\newcommand\dS{\mathrm{dS}}
\newcommand\ZZ{\mathbb{Z}}
\newcommand\codim{\operatorname{codim}}
\newcommand\Sym{\operatorname{Sym}}
\newcommand\End{\operatorname{End}}
\newcommand\Span{\operatorname{span}}
\newcommand\Ran{\operatorname{Ran}}
\newcommand\ep{\epsilon}
\newcommand\Cinf{\cC^\infty}
\newcommand\dCinf{\dot \cC^\infty}
\newcommand\CI{\cC^\infty}
\newcommand\dCI{\dot \cC^\infty}
\newcommand\Cx{\mathbb{C}}
\newcommand\Nat{\mathbb{N}}
\newcommand\dist{\cC^{-\infty}}
\newcommand\ddist{\dot \cC^{-\infty}}
\newcommand\pa{\partial}
\newcommand\Card{\mathrm{Card}}
\renewcommand\Box{{\square}}
\newcommand\Ell{\mathrm{Ell}}
\newcommand\WF{\mathrm{WF}}
\newcommand\WFh{\mathrm{WF}_\semi}
\newcommand\WFb{\mathrm{WF}_\bl}
\newcommand\Vf{\mathcal{V}}
\newcommand\Vb{\mathcal{V}_\bl}
\newcommand\Vz{\mathcal{V}_0}
\newcommand\Hb{H_{\bl}}
\newcommand\Ker{\mathrm{Ker}}
\newcommand\Range{\mathrm{Ran}}
\newcommand\Hom{\mathrm{Hom}}
\newcommand\Id{\mathrm{Id}}
\newcommand\sgn{\operatorname{sgn}}
\newcommand\ff{\mathrm{ff}}
\newcommand\tf{\mathrm{tf}}
\newcommand\esssupp{\operatorname{esssupp}}
\newcommand\supp{\operatorname{supp}}
\newcommand\vol{\mathrm{vol}}
\newcommand\Diff{\mathrm{Diff}}
\newcommand\Diffd{\mathrm{Diff}_{\dagger}}
\newcommand\Diffs{\mathrm{Diff}_{\sharp}}
\newcommand\Diffb{\mathrm{Diff}_\bl}
\newcommand\DiffbI{\mathrm{Diff}_{\bl,I}}
\newcommand\Diffbeven{\mathrm{Diff}_{\bl,\even}}
\newcommand\Diffz{\mathrm{Diff}_0}
\newcommand\Psih{\Psi_{\semi}}
\newcommand\Psihcl{\Psi_{\semi,\cl}}
\newcommand\Psib{\Psi_\bl}
\newcommand\Psibc{\Psi_{\mathrm{bc}}}
\newcommand\TbC{{}^{\bl,\Cx} T}
\newcommand\Tb{{}^{\bl} T}
\newcommand\Sb{{}^{\bl} S}
\newcommand\Lambdab{{}^{\bl} \Lambda}
\newcommand\zT{{}^{0} T}
\newcommand\Tz{{}^{0} T}
\newcommand\zS{{}^{0} S}
\newcommand\dom{\mathcal{D}}
\newcommand\cA{\mathcal{A}}
\newcommand\cB{\mathcal{B}}
\newcommand\cE{\mathcal{E}}
\newcommand\cG{\mathcal{G}}
\newcommand\cH{\mathcal{H}}
\newcommand\cU{\mathcal{U}}
\newcommand\cO{\mathcal{O}}
\newcommand\cF{\mathcal{F}}
\newcommand\cM{\mathcal{M}}
\newcommand\cQ{\mathcal{Q}}
\newcommand\cR{\mathcal{R}}
\newcommand\cI{\mathcal{I}}
\newcommand\cL{\mathcal{L}}
\newcommand\cK{\mathcal{K}}
\newcommand\cC{\mathcal{C}}
\newcommand\cX{\mathcal{X}}
\newcommand\cY{\mathcal{Y}}
\newcommand\cP{\mathcal{P}}
\newcommand\cS{\mathcal{S}}
\newcommand\cZ{\mathcal{Z}}
\newcommand\cW{\mathcal{W}}
\newcommand\Ptil{\tilde P}
\newcommand\ptil{\tilde p}
\newcommand\chit{\tilde \chi}
\newcommand\yt{\tilde y}
\newcommand\zetat{\tilde \zeta}
\newcommand\xit{\tilde \xi}
\newcommand\taut{{\tilde \tau}}
\newcommand\phit{{\tilde \phi}}
\newcommand\mut{{\tilde \mu}}
\newcommand\sigmat{{\tilde \sigma}}
\newcommand\sigmah{\hat\sigma}
\newcommand\zetah{\hat\zeta}
\newcommand\etah{\hat\eta}
\newcommand\loc{\mathrm{loc}}
\newcommand\compl{\mathrm{comp}}
\newcommand\reg{\mathrm{reg}}
\newcommand\GBB{\textsf{GBB}}
\newcommand\GBBsp{\textsf{GBB}\ }
\newcommand\bl{{\mathrm b}}
\newcommand{\sH}{\mathsf{H}}
\newcommand{\cte}{\digamma}
\newcommand\cl{\operatorname{cl}}
\newcommand\hsf{\mathcal{S}}
\newcommand\Div{\operatorname{div}}
\newcommand\hilbert{\mathfrak{X}}

\newcommand\Hh{H_{\semi}}

\newcommand\bM{\bar M}
\newcommand\Xext{X_{-\delta_0}}

\newcommand\xib{{\underline{\xi}}}
\newcommand\etab{{\underline{\eta}}}
\newcommand\zetab{{\underline{\zeta}}}

\newcommand\xibh{{\underline{\hat \xi}}}
\newcommand\etabh{{\underline{\hat \eta}}}
\newcommand\zetabh{{\underline{\hat \zeta}}}

\newcommand\zn{z}
\newcommand\sigman{\sigma}
\newcommand\psit{\tilde\psi}
\newcommand\rhot{{\tilde\rho}}

\newcommand\hM{\hat M}

\newcommand\Op{\operatorname{Op}}
\newcommand\Oph{\operatorname{Op_{\semi}}}

\newcommand\innr{{\mathrm{inner}}}
\newcommand\outr{{\mathrm{outer}}}
\newcommand\full{{\mathrm{full}}}
\newcommand\semi{\hbar}

\newcommand\elliptic{\mathrm{ell}}
\newcommand\diffordgen{k}
\newcommand\difford{2}
\newcommand\diffordm{1}
\newcommand\diffordmpar{1}
\newcommand\even{\mathrm{even}}
\newcommand\dimn{n}
\newcommand\dimnpar{n}
\newcommand\dimnm{n-1}
\newcommand\dimnp{n+1}
\newcommand\dimnppar{(n+1)}
\newcommand\dimnppp{n+3}
\newcommand\dimnppppar{n+3}

\newcommand\cSXp{\cS_{X_+}}
\newcommand\cSXm{\cS_{X_-}}
\newcommand\cSXpm{\cS_{X_\pm}}
\newcommand\cSXz{\cS_{X_0}}
\newcommand\cSXzb{\cS_{X_0,\past}}
\newcommand\cSXt{\cS_{\tilde X}}
\newcommand\cSXtb{\cS_{\tilde X,\past}}
\newcommand\cPXp{\cP_{X_+}}
\newcommand\cPXm{\cP_{X_-}}
\newcommand\cPXpm{\cP_{X_\pm}}
\newcommand\cPXz{\cP_{X_0}}
\newcommand\cPXzb{\cP_{X_0,\past}}
\newcommand\cPXzf{\cP_{X_0,\future}}
\newcommand\cPXt{\cP_{\tilde X}}
\newcommand\cPXtb{\cP_{\tilde X,\past}}

\newcommand\tPsb{\tilde P_{\sigma,\past}^{-1}}
\newcommand\tPssb{\tilde P_{\cdot,\past}^{-1}}
\newcommand\tPsf{\tilde P_{\sigma,\future}^{-1}}

\newcommand\past{\mathrm{past}}
\newcommand\future{\mathrm{future}}

\newcommand\xXp{x_{X_+}}
\newcommand\xXm{x_{X_-}}
\newcommand\xXpm{x_{X_\pm}}
\newcommand\xXz{x_{X_0}}
\newcommand\xXzp{x_{X_0}^+}
\newcommand\xXzm{x_{X_0}^-}
\newcommand\xXzpm{x_{X_0}^\pm}
\newcommand\xXd{x_{X_\bullet}}

\newcommand\aXp{a_{X_+}}
\newcommand\aXpp{a_{X_+}^+}
\newcommand\aXpm{a_{X_+}^-}
\newcommand\aXppm{a_{X_+}^{\pm}}
\newcommand\aXpmpm{a_{X_\pm}^{\pm}}
\newcommand\aXm{a_{X_-}}
\newcommand\aXmp{a_{X_-}^+}
\newcommand\aXmm{a_{X_-}^-}
\newcommand\aXmpm{a_{X_-}^{\pm}}
\newcommand\aXppz{a_{X_+,0}^+}
\newcommand\aXpmz{a_{X_+,0}^-}
\newcommand\aXppmz{a_{X_+,0}^{\pm}}
\newcommand\aXpmpmz{a_{X_\pm,0}^{\pm}}
\newcommand\aXz{a_{X_0}}
\newcommand\aXzp{a_{X_0}^+}
\newcommand\aXzm{a_{X_0}^-}
\newcommand\aXzpz{a_{X_0,0}^+}
\newcommand\aXzmz{a_{X_0,0}^-}
\newcommand\aXzpm{a_{X_0}^{\pm}}
\newcommand\aXzpmz{a_{X_0,0}^{\pm}}
\newcommand\aaXzm{a_{X_0,\past}}
\newcommand\aXzmpm{a_{X_0,\past}^{\pm}}
\newcommand\aXzmp{a_{X_0,\past}^{+}}
\newcommand\aXzmm{a_{X_0,\past}^{-}}
\newcommand\aXt{a_{\tilde X}}
\newcommand\aXtp{a_{\tilde X}^+}
\newcommand\aXtm{a_{\tilde X}^-}
\newcommand\aXtpm{a_{\tilde X}^\pm}
\newcommand\aXtpmz{a_{\tilde X,0}^\pm}
\newcommand\aaXtm{a_{\tilde X,\past}}
\newcommand\aXtmp{a_{\tilde X,\past}^+}
\newcommand\aXtmm{a_{\tilde X,\past}^-}
\newcommand\aXtmz{a_{\tilde X,\past}^0}
\newcommand\aXtmpm{a_{\tilde X,\past}^\pm}
\newcommand\aXtmpmsz{a_{\tilde X,\past,0}^\pm}
\newcommand\aXtXm{a_{\tilde X,X_-}}
\newcommand\aXtXp{a_{\tilde X,X_+}}
\newcommand\aXtXz{a_{\tilde X,X_0}}
\newcommand\aXtXzp{a_{\tilde X,X_0}^+}
\newcommand\aXtXzm{a_{\tilde X,X_0}^-}
\newcommand\aXtXzpm{a_{\tilde X,X_0}^\pm}
\newcommand\aXtXmp{a_{\tilde X,X_-}^+}
\newcommand\aXtXmm{a_{\tilde X,X_-}^-}
\newcommand\aXtXmpm{a_{\tilde X,X_-}^\pm}
\newcommand\aXtmpj{a_{\tilde X,\past,j}^+}
\newcommand\aXtmmj{a_{\tilde X,\past,j}^-}
\newcommand\aXtXmmj{a_{\tilde X,X_-,j}^-}
\newcommand\aXtXzmj{a_{\tilde X,X_0,j}^-}
\newcommand\aXtpsz{a_{\tilde X,0}^+}
\newcommand\aXtmsz{a_{\tilde X,0}^-}
\newcommand\aXtpmsz{a_{\tilde X,0}^{\pm}}

\newcommand\bXppz{b_{X_+,0}^+}
\newcommand\bXzpmz{b_{X_0,0}^{\pm}}
\newcommand\bXzpz{b_{X_0,0}^{+}}
\newcommand\bXzmz{b_{X_0,0}^{-}}
\newcommand\bXtpmz{b_{\tilde X,0}^{\pm}}
\newcommand\bXtpz{b_{\tilde X,0}^{+}}
\newcommand\bXtmz{b_{\tilde X,0}^{-}}

\newcommand\vXp{v_{X_+}}
\newcommand\vXpp{v_{X_+}^+}
\newcommand\vXpm{v_{X_+}^-}
\newcommand\vXppm{v_{X_+}^{\pm}}
\newcommand\vXpmpm{v_{X_\pm}^{\pm}}
\newcommand\vXm{v_{X_-}}
\newcommand\vXmp{v_{X_-}^+}
\newcommand\vXmm{v_{X_-}^-}
\newcommand\vXmpm{v_{X_-}^{\pm}}
\newcommand\vXz{v_{X_0}}
\newcommand\vXzp{v_{X_0}^+}
\newcommand\vXzm{v_{X_0}^-}
\newcommand\vXzpm{v_{X_0}^{\pm}}
\newcommand\vvXzm{v_{X_0,\past}}
\newcommand\vXzmpm{v_{X_0,\past}^{\pm}}
\newcommand\vXzmp{v_{X_0,\past}^{+}}
\newcommand\vXzmm{v_{X_0,\past}^{-}}
\newcommand\vXt{v_{\tilde X}}
\newcommand\vXtp{v_{\tilde X}^+}
\newcommand\vXtm{v_{\tilde X}^-}
\newcommand\vXtz{v_{\tilde X}^0}
\newcommand\vXtpm{v_{\tilde X}^\pm}
\newcommand\vvXtm{v_{\tilde X,\past}}
\newcommand\vXtmp{v_{\tilde X,\past}^+}
\newcommand\vXtmm{v_{\tilde X,\past}^-}
\newcommand\vXtmz{v_{\tilde X,\past}^0}
\newcommand\vXtmpm{v_{\tilde X,\past}^\pm}
\newcommand\vXtXm{v_{\tilde X,X_-}}
\newcommand\vXtXp{v_{\tilde X,X_+}}
\newcommand\vXtXz{v_{\tilde X,X_0}}
\newcommand\vXtXzp{v_{\tilde X,X_0}^+}
\newcommand\vXtXzm{v_{\tilde X,X_0}^-}
\newcommand\vXtXzpm{v_{\tilde X,X_0}^\pm}
\newcommand\vXtXmp{v_{\tilde X,X_-}^+}
\newcommand\vXtXmm{v_{\tilde X,X_-}^-}
\newcommand\vXtXmpm{v_{\tilde X,X_-}^\pm}
\newcommand\vXtXpmz{v_{\tilde X,X_+,0}^-}

\newcommand\wXzp{w_{X_0}^+}
\newcommand\wXzm{w_{X_0}^-}

\newcommand\uXp{u_{X_+}}
\newcommand\uXm{u_{X_-}}
\newcommand\uXz{u_{X_0}}
\newcommand\uXt{u_{\tilde X}}
\newcommand\uXtXm{u_{\tilde X,X_-}}
\newcommand\uXtXp{u_{\tilde X,X_+}}
\newcommand\uXtXz{u_{\tilde X,X_0}}

\newcommand\fXp{f_{X_+}}
\newcommand\fXpmpm{f_{X_\pm}^{\pm}}
\newcommand\fXpp{f_{X_+}^+}
\newcommand\fXpm{f_{X_+}^-}
\newcommand\fXppm{f_{X_+}^{\pm}}
\newcommand\fXm{f_{X_-}}
\newcommand\fXmp{f_{X_-}^+}
\newcommand\fXmm{f_{X_-}^-}
\newcommand\fXmpm{f_{X_-}^{\pm}}
\newcommand\fXz{f_{X_0}}
\newcommand\fXzp{f_{X_0}^+}
\newcommand\fXzm{f_{X_0}^-}
\newcommand\fXzpm{f_{X_0}^{\pm}}
\newcommand\fXt{f_{\tilde X}}
\newcommand\fXtp{f_{\tilde X}^+}
\newcommand\fXtm{f_{\tilde X}^-}
\newcommand\ffXtm{f_{\tilde X,\past}}
\newcommand\fXtmp{f_{\tilde X,\past}^+}
\newcommand\fXtmm{f_{\tilde X,\past}^-}
\newcommand\fXtmz{f_{\tilde X,\past}^0}
\newcommand\fXtmpm{f_{\tilde X,\past}^\pm}

\newcommand\Xd{X_\bullet}
\newcommand\Xde{X_{\bullet,\even}}

\setcounter{secnumdepth}{3}
\newtheorem{lemma}{Lemma}[section]
\newtheorem{prop}[lemma]{Proposition}
\newtheorem{thm}[lemma]{Theorem}
\newtheorem{cor}[lemma]{Corollary}
\newtheorem{result}[lemma]{Result}
\newtheorem*{thm*}{Theorem}
\newtheorem*{prop*}{Proposition}
\newtheorem*{cor*}{Corollary}
\newtheorem*{conj*}{Conjecture}
\numberwithin{equation}{section}
\theoremstyle{remark}
\newtheorem{rem}[lemma]{Remark}
\newtheorem*{rem*}{Remark}
\theoremstyle{definition}
\newtheorem{Def}[lemma]{Definition}
\newtheorem*{Def*}{Definition}

\newcommand{\mar}[1]{{\marginpar{\sffamily{\scriptsize #1}}}}
\newcommand\av[1]{\mar{AV:#1}}

\renewcommand{\theenumi}{\roman{enumi}}
\renewcommand{\labelenumi}{(\theenumi)}

\title[Resolvents, Poisson operators and scattering matrices]{Resolvents, Poisson operators and scattering matrices on
asymptotically hyperbolic  and de Sitter spaces}
\author[Andras Vasy]{Andr\'as Vasy}
\address{Department of Mathematics, Stanford University, CA 94305-2125, USA}

\email{andras@math.stanford.edu}


\date{March 22, 2013.}
\thanks{The author gratefully
  acknowledges partial support from the NSF under grant number  
  DMS-1068742.}

\begin{abstract}
We describe how the global operator induced on the boundary of an
asymptotically Minkowski space links two asymptotically hyperbolic
spaces and an asymptotically de Sitter space, and compute the scattering
operator of the linked problem in terms of the scattering operator of
the constituent pieces.
\end{abstract}

\maketitle

\section{Introduction}
In \cite{Vasy-Dyatlov:Microlocal-Kerr} and \cite{Vasy:Microlocal-AH}
new methods were introduced to study the spectral and scattering theory
of the Laplacian on
asymptotically hyperbolic spaces and of the d'Alembertian on
asymptotically de Sitter spaces $(X,g)$. Concretely, examples of these
spaces showed up as
boundary values of a one higher dimensional space $\tilde M$ equipped with a
Lorentzian metric $\tilde g$, which was either a
blown-up version of de Sitter space, or a Kerr-de Sitter type space
(which is a generalization of the former), or a Minkowski
space. However, the analysis could be done (as long as $g$ was a
so-called even metric) without introducing a one
higher dimensional space, by extending across the boundary of the
conformal compactification $\overline{X}$, with a new smooth structure
(the defining function of the boundary replaced by its square, hence
the relevance of evenness) in a suitable manner. This was done
systematically and in full generality in \cite{Vasy:Microlocal-AH} for the case of an asymptotically
hyperbolic space, with complex absorption introduced in the de Sitter
region, and was extended to differential forms in
\cite{Vasy:Analytic-forms}.

Here we recall that a compact $n$-dimensional manifold with boundary,
$\overline{X}$, with interior $X$ equipped with a metric $g$,
is asymptotically hyperbolic, resp.\ de
Sitter, if $g=\frac{\hat g}{x^2}$ where $\hat g$ is a $\CI$
Riemannian, resp.\ Lorentzian (of signature $(1,n-1)$), metric on
$\overline{X}$, with $\hat g(dx,dx)|_{x=0}=1$, for a boundary defining
function $x$. In the Lorentzian setting one also assumes that the
boundary $Y$ of $\overline{X}$ is of the form $Y=Y_+\cup Y_-$, with
$Y_\pm$ unions of connected components, and all
(null-)bicharacteristics\footnote{By bicharacteristics we always mean null-bicharacteristics.} $\gamma(t)$, or equivalently null-geodesics, of $g$
defined over $\RR$ in
$X$ tend to $Y_+$ as the parameter $t\to+\infty$ and to $Y_-$ as
$t\to-\infty$, or vice versa. (This implies global hyperbolicity and
that $\overline{X}$ is diffeomorphic to $[-1,1]\times Y_+$. Further,
the null-bicharacteristics, and hence the null-geodesics, are simply
reparameterized by a conformal factor, such as $x^2$, away from where
it is singular/vanishes, i.e.\ away from the boundary. Correspondingly,
the
requirement on the bicharacteristics is equivalent to maximally
extended bicharacteristics of $\hat g$ being defined over compact
intervals, taking values over $Y_+$ at one endpoint and $Y_-$ at the
other.)

As shown by Graham and Lee \cite{Graham-Lee:Einstein} in
the Riemannian case, and by a similar argument in the Lorentzian case,
there is then a product decomposition near the
boundary $Y_{y}$ of $\overline{X}$
such that
$$
g=\frac{dx^2+\tilde h(x,y,dy)}{x^2}.
$$
If this decomposition can be chosen so that $\tilde h$ is even in $x$,
i.e.\ $\tilde h=h(x^2,y,dy)$, we call $g$ even. This is
equivalent to saying that $h$ is $\CI$ on $\overline{X}_{\even}$, the even
version of $\overline{X}$, which is $\overline{X}$ as a topological manifold, but
the $\CI$ structure is changed so that $\mu=x^2$ is the new defining
function of the boundary.

Returning to the general discussion, there are natural settings, namely asymptotically
Minkowski spaces, in which combinations of asymptotically
de Sitter and asymptotically hyperbolic spaces appear linked in
interesting ways. A class of asymptotically Minkowski spaces
$(\tilde M,\tilde g)$, with $\overline{\tilde M}$ being the
compactification of $\tilde M$
with respect to which $\tilde g$ has appropriate properties, was
introduced by Baskin, Vasy and Wunsch in
\cite{Baskin-Vasy-Wunsch:Radiation}, but as here we think of $\tilde M$ as a
motivation for linking two copies $(X_+,g_+)$ and $(X_-,g_-)$ of asymptotically hyperbolic spaces
(in case of Minkowski space, the quotient of the interior of the future and past light
cones by the $\RR^+$-action) and an asymptotically de Sitter space $(X_0,g_0)$ (in
case of Minkowski space, the quotient of the exterior of the light
cones by the $\RR^+$-action) rather than the main object of interest,
this general class is not directly important here; the important
aspect is the asymptotic behavior of its elements at infinity. In
particular, we may assume that $\tilde M$ is replaced by a new manifold
equipped with an $\RR_+$-action, denoted by $M$, indeed is of
the form $\RR^+_\rho\times\tilde X$, with $\tilde
X=\pa\overline{\tilde M}$; here $\tilde\rho=\rho^{-1}$ is a boundary defining function
of $\overline{\tilde M}$ (thus the boundary of $\overline{\tilde M}$ is where $\rho$ is
infinite). Within
$$
\tilde
X=\overline{X_+}\cup\overline{X_-}\cup\overline{X_0},
$$
the boundaries
of $\overline{X_+}$, resp.\ $\overline{X_-}$, and the future, resp.\ past
boundaries, $\pa_+\overline{X_0}$, resp.\ $\pa_-\overline{X_0}$, are identified.
Mellin transforming (the conjugate by $\rho^{(n-1)/2}$ of) $\rho^2\Box_{\tilde g}$ induces a family of
operators $\tilde P_\sigma$ on $\tilde X$; we refer to this as the family of {\em
  global} operators (on $\tilde X$). On the other hand, a differently normalized
Mellin transform over the smaller domains $X_\pm$ and $X_0$ (which
becomes singular at the boundary of these domains) induces
the spectral families of asymptotically hyperbolic ($X_\pm$)
Laplacians and asymptotically de Sitter ($X_0$) d'Alembertians; we
call these the {\em constituent} operators.
Starting with \cite{Vasy-Dyatlov:Microlocal-Kerr} and
\cite{Vasy:Microlocal-AH}, continued in \cite{Vasy:Analytic-forms} and
\cite{Baskin-Vasy-Wunsch:Radiation}, some aspects of the connection
being the global and constituent operators were explored.
In this paper we show how the global operator on $\tilde X$
links the three constituent operators explicitly. In particular, we relate the
scattering operators (or matrices) of the constituent operators to the
global scattering operator. We remark here that given either an
asymptotically hyperbolic space or an asymptotically de Sitter space,
the spaces $\tilde X$ and $M$ can always be constructed (after
possibly taking two copies of the asymptotically de Sitter space); see Section~\ref{sec:conformal}.

To make this concrete,
the relationship
between the scattering operators
\begin{equation*}\begin{aligned}
&\cSXtb(\sigma):\CI(\pa X_+)\oplus \CI(\pa X_+)\to \CI(\pa X_-)\oplus
\CI(\pa X_-)\ \text{on}\ \tilde X,\\
&\cSXp(\sigma):\CI(\pa X_+)\to\CI(\pa X_+)\ \text{on}\ X_+,\\
&\cSXm(\sigma) :\CI(\pa X_-)\to\CI(\pa X_-)\ \text{on}\ X_-,\  \text{and}\\
&\cSXzb(\sigma) :\CI(\pa_+ X_0)\oplus \CI(\pa_+ X_0)\to \CI(\pa_-
X_0)\oplus \CI(\pa_- X_0)\ \text{on}\ X_0,
\end{aligned}\end{equation*}
(recall that $\pa_+ X_0=\pa X_+$
and $\pa_- X_0=\pa X_-$), defined in
Definitions~\ref{def:Poisson-Xt}, \ref{def:Poisson-X+} and
\ref{def:Poisson-X0} respectively, is given by the following theorem:

\begin{thm}(See Theorem~\ref{thm:S-matrix} and Corollary~\ref{cor:FIO}.)
For $\sigma\notin\imath\ZZ$, if $\sigma$ is not a pole of the inverse
$\tPsb$
of the global
operator $\tilde P_\sigma$ on $\tilde X$ (acting between function
spaces discussed at the end of Section~\ref{sec:conformal}, which
amounts to solving the backwards, or past-oriented problem,
propagating regularity towards $\pa_-X_0$)
then
\begin{equation*}\begin{aligned}
&\cSXtb(\sigma)=\begin{bmatrix}
e^{-\pi \sigma}&e^{\pi\sigma}\\1&1\end{bmatrix}^{-1}
\begin{bmatrix}\Id&0\\0&\cSXm(-\sigma)\end{bmatrix}
\cSXzb(\sigma)
\begin{bmatrix}\Id&0\\0&\cSXp(\sigma)\end{bmatrix}\begin{bmatrix}e^{-\pi\sigma}&e^{\pi\sigma}\\
1&1\end{bmatrix},
\end{aligned}\end{equation*}
i.e.\ apart from integer issues corresponding to the matrices with
$e^{\pi\sigma}$ terms, $\cSXtb(\sigma)$ is essentially the product of
$\cSXpm(\pm\sigma)$ and $\cSXzb(\sigma)$.

Furthermore, $\cSXtb(\sigma)$ is an elliptic Fourier integral operator of order
$0$ associated to the (rescaled or limiting) null-geodesic flow on $X_0$, from
$\pa_+X_0$ to $\pa_-X_0$, with principal symbol as stated in Corollary~\ref{cor:FIO}.
\end{thm}

The Fourier integral operator statement is proved using results of
Joshi and S\'a Barreto \cite{Joshi-Sa-Barreto:Inverse} on the
scattering matrix on asymptotically hyperbolic spaces being a
pseudodifferential operator, and of the author that the scattering
operator on asymptotically de Sitter spaces is a Fourier integral
operator associated to the null-geodesic flow
\cite{Vasy:De-Sitter}. Proving the FIO property of $\cSXtb(\sigma)$ intrinsically on
$\tilde X$ is a subject of current work with Nick Haber.

We also describe $\tPsb$ in terms of the resolvents and
Poisson operators in terms of the constituent pieces, see
Theorem~\ref{thm:parts-to-whole}.

In the whole paper we consider the operators acting on functions to
simplify the notation. In \cite{Vasy:Analytic-forms} the setup was
translated to differential forms, and at the cost of somewhat more
complicated notation/asymptotics (distinguishing closed and co-closed
forms), one could work with the form bundles. However, while the
methods of \cite{Joshi-Sa-Barreto:Inverse} and \cite{Vasy:De-Sitter}
work on the form bundles, the analysis there was not carried out in
that setting, so the extension of the FIO statement would require
additional work.

The plan of this paper is the following. In Section~\ref{sec:Minkowski} we recall
how the spaces are linked via the Mellin transform in the case of
Minkowski space. Motivated by this, in
Section~\ref{sec:conformal} we show that given an asymptotically de
Sitter or asymptotically hyperbolic space, one can construct an
asymptotically Minkowski space so that via the Mellin transform one
obtains a family of operators related to the spectral family of the
individual spaces which links them together. In
Section~\ref{sec:relationship} we establish the relationship between
these operators as well as their Poisson operators and scattering operators.

I am very grateful to Jared Wunsch, Dean Baskin, Richard Melrose, Rafe
Mazzeo, Maciej Zworski and Steve Zelditch for
interesting discussions and their encouragement.

\section{Minkowski space, hyperbolic space and de Sitter
  space}\label{sec:Minkowski}
In this section we connect the analysis of the
Laplacians/d'Alembertians
on Minkowski, hyperbolic and
de Sitter spaces. This connection has a direct extension, with simple
modifications, to the general asymptotically hyperbolic/de Sitter
setting, considered in the next section. Here we follow
\cite{Vasy:Analytic-forms}, which considered differential forms, in the
setup, but for the sake of the simplicity of notation we work in the
scalar setting (but this is completely unimportant).

The starting point of analysis is the manifold $\RR^{n+1}$, or rather
$\RR^{n+1}\setminus o$, which is equipped with an $\RR^+$-action given
by dilations: $(\lambda,z)\mapsto \lambda z$.
A transversal to this action is, as a differentiable manifold,
$\sphere^n$, which may
be considered
as the unit sphere with respect to the Euclidean metric, though the
metric properties are not important here (since we are interested in
the Minkowski metric after all).
Thus, writing $(z_1,\ldots,z_{n+1})$ as
the coordinates, let
$$
dz_1^2+\ldots+dz_n^2+dz_{n+1}^2,
$$
be the {\em Euclidean} metric, and
 let $\rho$ be the Euclidean distance function on
$\RR^{n+1}$ from the origin,
namely
$$
\rho=(z_1^2+\ldots+z_n^2+z_{n+1}^2)^{1/2}.
$$
Then $\sphere^n$ is the $1$-level set of $\rho$.
One can identify $\RR^{n+1}\setminus\{0\}$ via the Euclidean polar
coordinate map with
$\RR^+_\rho\times\sphere^n$, namely the map is
$\RR^+_\rho\times\sphere^n \ni(\rho,y)\mapsto \rho y\in \RR^{n+1}\setminus\{0\}$.

The Minkowski metric is given by
$$
\tilde g=dz_{n+1}^2-(dz_1^2+\ldots+dz_n^2),
$$
and we also consider the Minkowski distance function $r$.
Thus, away from the light cone, where $z_{n+1}^2=z_1^2+\ldots+z_n^2$,
let
$$
r=|z_{n+1}^2-(z_1^2+\ldots+z_n^2)|^{1/2}.
$$
To analyze $\Box_{\tilde g}$,
we conjugate $\rho^2\Box_{\tilde g}$ by the Mellin transform $\cM_\rho$
on $\RR^+_\rho\times\sphere^n$, identified with
$\RR^{n+1}\setminus\{0\}$ as above.
The so-obtained operator,
$$
\tilde P_{0,\sigmat}=\cM_\rho \rho^{2}\Box_{\tilde g}\cM_\rho^{-1}\in\Diff^2(\sphere^n),
$$
with $\sigmat$ the Mellin dual parameter,
fits into the framework of \cite{Vasy-Dyatlov:Microlocal-Kerr} and
\cite{Vasy:Microlocal-AH}, see
\cite[Section~5]{Vasy-Dyatlov:Microlocal-Kerr}.
As an aside, we remark that it will be convenient to shift the
Mellin parameter, or equivalently conjugate $\Box_{\tilde g}$ by a
power of $\rho$; this is the reason for adding
the cumbersome subscript $0$ to $\tilde P_{0,\sigmat}$ presently.

While so far we explained why the Minkowski wave operator can
be analyzed by means of \cite{Vasy-Dyatlov:Microlocal-Kerr} and
\cite{Vasy:Microlocal-AH}, we still need to connect this to
asymptotically hyperbolic and de Sitter spaces.
But in the region in $\sphere^n$ corresponding to the interior of the future
light cone, which can be identified with the hyperboloid
$$
\HH^n:\ z_{n+1}^2-(z_1^2+\ldots+z_n^2)=1,\ z_{n+1}>0,
$$
via the $\RR^+$-quotient,
one can also consider the Mellin transform of $r^2\Box_{\tilde g}$
with respect to the decomposition $\RR^+_r\times\HH^n$, to
get
$$
P_\sigmat=\cM_r r^2\Box_{\tilde g}\cM_r^{-1}\in\Diff^2(\HH^n).
$$
(There is a similar setup for the second copy of $\HH^n$ in the past
light cone, where $z_{n+1}<0$.)
Now, $P_\sigmat$ is not well-behaved at the boundary of the
future light cone, but it is closely related to $\tilde P_\sigmat$. Namely, if
we use coordinates
$$
y_j=\frac{z_j}{z_{n+1}},\ j=1,\ldots,n,
$$
on the sphere away from the equator $z_{n+1}=0$,
$$
r=F(y)\rho,\ F(y)=\sqrt{\frac{1-|y|^2}{1+|y|^2}}.
$$
Note that $F^2$ is a smooth function on $\sphere^n$ near (its
intersection with) the light cone which vanishes non-degenerately at
the light cone. On the other hand, the Poincar\'e ball model
$\overline{\HH^n}$ of
$\HH^n$ arises by regarding it as a graph over $\RR^n$ in
$\RR^n\times\RR$, and compactifying $\RR^n$ radially (or geodesically)
to a ball, with boundary defining
function, say, $(z_1^2+\ldots+z_n^2)^{-1/2}$, or, $\rho^{-1}$ -- these
two differ by a smooth positive multiple on $\overline{\HH^n}$. As
$r=1$ on $\HH^n$, this means that $F$ is a valid boundary defining
function in the Poincar\'e model, in contrast with the natural $F^2$
defining function of the light cone. In particular, with $\hat y_j$,
$j=1,\ldots,n-1$, denoting local coordinates on $\sphere^{n-1}$, identified
with $\pa \overline{\HH^n}$, hence the light cone at infinity is
identified with $\sphere^{n-1}$,
pulling back the Minkowski
metric to $\HH^n$, which by definition yields the hyperbolic metric,
a straightforward calculation yields that that
\begin{equation}\label{eq:hyp-metric-form-F}
g=\frac{(dF)^2}{F^2(1-F^2)}+\frac{1-F^2}{2F^2}h(\hat y,d\hat y),
\end{equation}
with $h$ the round metric on the sphere; this satisfies $F^2g$ being a
smooth metric up to the boundary, $F=0$ (with a polar coordinate
singularity at $F=1$; $F$ and $\hat y$ are not valid coordinates
there, though $F$ is still $\CI$ near $F=1$, and the metric is still
$\CI$ there as
well, as can be seen by using valid coordinates), with the coefficients even functions of $F$. The metric $g$ can be put in
the normal form $g=\frac{dx^2+h}{x^2}$ by letting
$x=\frac{F}{1+\sqrt{1-F^2}}$, which is an equivalent boundary defining
function, but this is not necessary here.

Since
$$
\cM_\rho f(\sigmat,y)=\int_0^\infty \rho^{-\imath\sigmat} f\,\frac{d\rho}{\rho},
$$
with a similar formula for $\cM_r$, we have, if we identify $\HH^n$
with an open subset of $\sphere^n$ (the interior of the future light
cone),
\begin{equation}\label{eq:relate-Mellin}
\cM_\rho\rho^2\Box_{\tilde
 g}\cM_\rho^{-1}(\sigmat)=F^{\imath\sigmat-2}\cM_r r^2\Box_{\tilde g}\cM_r^{-1}F^{-\imath\sigmat}.
\end{equation}

We next compute $\cM_r r^2\Box_{\tilde g}\cM_r^{-1}$; this is feasible
since $\RR^+\times\HH^n$ is an orthogonal decomposition relative to
$\tilde g$. Concretely, the Minkowski metric is
$$
\tilde g=dr^2-r^2 g,
$$
where $g$ is the hyperbolic metric, since by definition the hyperbolic
metric {\em is} the {\em negative} of the restriction of the Minkowski metric to the
hyperboloid $\HH^n$. This is a\footnote{Lorentzian, but this does not
  affect these computations.} conic metric, whose Laplacian is
\begin{equation}\label{eq:hyp-cap-Mink}
\Box_{\tilde g}=-r^{-2}\Delta_X-r^{-n}\pa_r r^{n}\pa_r,
\end{equation}
(cf.\ \cite[Equation~(3.8)]{Cheeger:Spectral} for the form version of
the computation). Rewriting this as
$$
r^2\Box_{\tilde g}=-\Delta_X-r^{-n+1}(r\pa_r) r^{n-1}(r\pa_r)=-\Delta_X-(r\pa_r+n-1)(r\pa_r),
$$
the Mellin transform of $r^2\Box_{\tilde g}$ with
respect to $r$ is
\begin{equation*}\begin{aligned}
&\cM_r r^2\Box_{\tilde
  g}\cM_r^{-1}(\tilde\sigma)=-\Delta_X-(\imath\sigmat+n-1)(\imath\sigmat)\\
&\qquad=-\Delta_X+(\sigmat-\imath(n-1))\sigmat=-\Delta_X+(\sigmat-\imath(n-1)/2)^2+(n-1)^2/4,
\end{aligned}\end{equation*}
which shows that it is useful to introduce
$\sigma=\sigmat-\imath(n-1)/2$, corresponding to the conjugation
$$
\cM_r r^{(n-1)/2}r^2\Box_{\tilde
  g}r^{-(n-1)/2}\cM_r^{-1}(\sigma)=-\Delta_X+\sigma^2+(n-1)^2/4.
$$
We remark that \eqref{eq:relate-Mellin} becomes
\begin{equation}\begin{aligned}\label{eq:relate-Mellin-mod}
&\cM_\rho\rho^2\rho^{(n-1)/2}\Box_{\tilde
 g}\rho^{-(n-1)/2}\cM_\rho^{-1}(\sigma)\\
&\qquad=F^{\imath\sigma-(n-1)/2-2}\cM_r
r^{(n-1)/2}r^2\Box_{\tilde
  g}r^{-(n-1)/2}\cM_r^{-1}F^{-\imath\sigma+(n-1)/2}\\
&\qquad=F^{\imath\sigma-(n-1)/2-2}(-\Delta_X+\sigma^2+(n-1)^2/4) F^{-\imath\sigma+(n-1)/2}.
\end{aligned}\end{equation}

We now replace $\HH^n$ with $\dS^n$ in our considerations.
Thus, we work in the region in $\sphere^n$ corresponding to the exterior of the
future and past light cones (the `equatorial belt'), which can be identified with the hyperboloid
$$
\dS^n:\ z_{n+1}^2-(z_1^2+\ldots+z_n^2)=-1,
$$
via the $\RR^+$-quotient. Now
$$
\tilde g=-dr^2+g,
$$
where $g$ is the de Sitter metric. We next
consider the Mellin transform of $r^2\Box_{\tilde g}$
with respect to the decomposition $\RR^+_r\times\dS^n$, to
get
$$
P_\sigmat=\cM_r r^2\Box_{\tilde g}\cM_r^{-1}\in\Diff^2(\dS^n).
$$
Note that
$$
\Box_{\tilde g}=r^{-2}\Box_X+r^{-n}\pa_r r^{n}\pa_r,
$$
in analogy with \eqref{eq:hyp-cap-Mink}, so the Mellin transform of
$r^{(n-1)/2}r^2\Box_{\tilde g}r^{-(n-1)/2}$ with respect to $r$ is
$$
\cM_r r^{(n-1)/2}r^2\Box_{\tilde
  g}r^{-(n-1)/2}\cM_r^{-1}(\sigma)=\Box_X-\sigma^2-(n-1)^2/4.
$$

We can relate this to the spherical Mellin transform by completely analogous
arguments as in the case of $\HH^n$, except that $F$ is replaced by
$$
\tilde F=\sqrt{\frac{|y|^2-1}{|y|^2+1}}=\sqrt{\frac{1-|y|^{-2}}{1+|y|^{-2}}}.
$$
In principle this works only
away from the equator (where one could
use $y$ as coordinates); to see that this in fact works globally,
one should use Euclidean polar coordinates $|z'|$ and $\hat
y=\frac{z'}{|z'|}$ in $\RR^n_{z'}$, and use $|y|^{-1}=\frac{z_{n+1}}{|z'|}$ and
$\hat y$ in $(-1,1)\times\sphere^{n-1}$; the second expression for
$\tilde F$ now shows the desired smooth behavior on $\dS^n$.
Thus,
\begin{equation}\begin{aligned}\label{eq:relate-Mellin-dS}
&\cM_\rho\rho^{(n-1)/2}\rho^2\Box_{\tilde
 g}\rho^{-(n-1)/2}\cM_\rho^{-1}(\sigma)\\
&\qquad=\tilde F^{\imath\sigma-(n-1)/2-2}\cM_r
r^2\Box_{\tilde g}\cM_r^{-1}\tilde F^{-\imath\sigma+(n-1)/2}\\
&\qquad=\tilde F^{\imath\sigma-(n-1)/2-2}(\Box_X-\sigma^2-(n-1)^2/4)\tilde F^{-\imath\sigma+(n-1)/2}
\end{aligned}\end{equation}

\section{Asymptotically Minkowski spaces}\label{sec:conformal}
We now extend the results to the operators induced on the boundary at
infinity of general asymptotically Minkowski
spaces; we further show below how these spaces arise from asymptotically
hyperbolic or de Sitter spaces in a natural way.
Since for us it is the boundary behavior that matters (rather
than the potentially complicated bicharacteristic flow in the interior), it is
convenient to set this up as a homogeneous metric (of degree $2$) on
$\RR^+\times \tilde X$, where $\tilde X$ is a compact manifold; for general
Lorentzian scattering metrics in the sense of
\cite{Baskin-Vasy-Wunsch:Radiation} this is the model at the boundary
of the compactified Lorentzian manifold (thus, we do not need the full
Lorentzian scattering metric setup of
\cite{Baskin-Vasy-Wunsch:Radiation}). Thus, as in
\cite{Baskin-Vasy-Wunsch:Radiation}, but using the product structure,
consider Lorentzian metrics of the form
$$
\tilde g=v\frac{d\rhot^2}{\rhot^4}-\Big(\frac{d\rhot}{\rhot^2}\otimes\frac{\alpha}{\rhot}+\frac{\alpha}{\rhot}\otimes
\frac{d\rhot}{\rhot^2}\Big)-\frac{\check g}{\rhot^2}
$$
where $\rhot=\rho^{-1}$ is the defining function of the boundary at
infinity (so is homogeneous of degree $-1$), $v\in\CI(\tilde X)$, $\alpha$ a
$\CI$ one-form on $\tilde X$, $\alpha|_{v=0}=\frac{1}{2}\,dv$, $\check g$ a
symmetric $\CI$ 2-cotensor on $\tilde X$ which is positive definite on the
annihilator of $dv$; in terms of
$\rho$ this takes the form
\begin{equation}\label{eq:asymp-Mink-homog-form}
\tilde g=v\,d\rho^2+\rho\Big(d\rho\otimes\alpha+\alpha\otimes
d\rho\Big)-\rho^2\check g.
\end{equation}

Such a metric gives rise to an asymptotically hyperbolic manifold
(with multiple connected components under the further assumptions we
make below) in $v>0$, and an asymptotically de-Sitter manifold in
$v<0$ (without the full dynamical hypotheses on these).

To see how the spectral family of the Laplacian, resp.\ the
d'Alembertian, of an {\em even} metric $g=g_\bullet$ on $X=\Xd$ (with
compactification $\overline{\Xd}$), fits into an asymptotically Minkowski
framework, first consider the operator
\begin{equation}\begin{aligned}\label{eq:AH-spect}
&P_\sigma=-\Delta_{\Xd}+\sigma^2+\Big(\frac{n-1}{2}\Big)^2,
\end{aligned}\end{equation}
resp.
\begin{equation}\begin{aligned}\label{eq:dS-spect}
&P_\sigma=\Box_{\Xd}-\sigma^2-\Big(\frac{n-1}{2}\Big)^2,
\end{aligned}\end{equation}
on the space $\Xd$, where $\bullet$ denotes a subscript, such as $+$
or $0$ below.
With $\overline{\Xde}$ the even version of $\overline{\Xd}$, and with
$\xXd$ a boundary defining function of $\overline{\Xd}$, we modify this
to the operator
\begin{equation}\label{eq:P-rel-Xde}
\tilde P_\sigma|_{\Xde}=\xXd^{\imath\sigma-(n-1)/2-2} P_\sigma \xXd^{-\imath\sigma+(n-1)/2},
\end{equation}
which one now checks is the restriction of an operator $\tilde P_\sigma$ defined on an
extension $\tilde X$ of $\Xde$ across $Y=\pa\Xde$, and satisfying the
requirements of \cite{Vasy-Dyatlov:Microlocal-Kerr} and
\cite{Vasy:Microlocal-AH}. This was checked explicitly in
\cite{Vasy:Microlocal-AH}. Note that at the level of the principal
symbol, given by the dual metric {\em function},
this means that $x^{-2}G$
extends smoothly to $T^*\tilde X$, which is automatic for an
even asymptotically hyperbolic metric. One does need to check the
behavior of the lower order terms (which {\em would} be singular
without the conjugation by $\xXd^{-\imath\sigma+(n-1)/2}$, while for
the principal symbol the latter does not matter), but this was again done in \cite{Vasy:Microlocal-AH}.

A different way of proceeding is via extending the metric $g=g_\bullet$ to an
ambient metric, playing the role of the Minkowski metric, which is
homogeneous of degree $2$. Thus, one considers
$M=\RR^+_{\rho}\times\tilde X$, as well as $\RR^+_r\times \Xd$, with
$\bullet=\pm$ for the asymptotically hyperbolic spaces, and with
$r=\xXpm\rho$, so $F=\xXpm$ in the Minkowski setting.
We note, however, that while with $F$ defined above
in the Minkowski setting, the hyperbolic metric has some higher order
(in $x=\xXpm$)
$dx^2=d\xXpm^2$ terms in view of \eqref{eq:hyp-metric-form-F},
these do not affect properties of the extension across $\xXpm=0$. On
$\RR^+_r\times \Xd$ the analogue of the Minkowski metric is
$$
\tilde g=dr^2-r^2 g=r^2\Big(\frac{dr^2}{r^2}-g\Big)=\rho^2\Big(\xXpm^2\Big(\frac{d\rho}{\rho}+\frac{d\xXpm}{\xXpm}\Big)^2-\xXpm^2g\Big).
$$
Substituting the form of $g$ and writing $\xXpm^2=\mu$,
\begin{equation}\label{eq:extend-AH-to-Mink}
\tilde g=\rho^2\Big(\mu
\frac{d\rho^2}{\rho^2}+\frac{1}{2}\Big(\frac{d\rho}{\rho}\otimes
d\mu+d\mu\otimes \frac{d\rho}{\rho}\Big)-h(\mu,\hat y,d\hat y)\Big).
\end{equation}
But now the desired extension is immediate to a neighborhood of
$\Xde$ in $\tilde X$ (which is
all that is required for the analysis if one uses complex absorption
as in
\cite{Vasy-Dyatlov:Microlocal-Kerr,Vasy:Microlocal-AH,Vasy:Analytic-forms}),
by simply extending $h$
smoothly to a neighborhood (i.e.\ from $\mu\geq 0$ to $\mu$ near $0$).
This is easily checked to be Lorentzian, and indeed a special
case\footnote{This assumes that
one ignores the interior of the space carrying a Lorentzian scattering
metric; more precisely it is a special case of the restriction of a
Lorentzian scattering metric to a neighborhood of the boundary of the
compactification of the space.}
of the scattering metrics of \cite{Baskin-Vasy-Wunsch:Radiation} in
view of \eqref{eq:asymp-Mink-homog-form}. Notice that the metric in
$\mu<0$ takes the form, with $\mu=-\xXz^2$,
\begin{equation*}\begin{aligned}
\tilde g&=
\rho^2\Big(-\xXz^2
\frac{d\rho^2}{\rho^2}-\xXz^2\Big(\frac{d\rho}{\rho}\otimes
\frac{d\xXz}{\xXz}+\frac{d\xXz}{\xXz}\otimes \frac{d\rho}{\rho}\Big)-h(-\xXz^2,\hat y,d\hat y)\Big)\\
&=
\rho^2\Big(-\xXz^2\Big(\frac{d\rho}{\rho}+\frac{d\xXz}{\xXz}\Big)^2+\xXz^2g_{X_0}\Big),
\end{aligned}\end{equation*}
with
\begin{equation}\label{eq:induced-gXz}
g_{X_0}=\frac{d\xXz^2-h(-\xXz^2,\hat y,d\hat y)}{\xXz^2},
\end{equation}
i.e.\ $g_{X_0}$ is asymptotically de Sitter, with cross-section metric
given by $h(-\xXz^2,\hat y,d\hat y)$ rather than $h(\xXz^2,\hat
y,d\hat y)$, i.e.\ it is
the extension of $h$ in the first argument across $0$ that enters into
$g_{X_0}$.

The analogous construction also works on asymptotically de Sitter
spaces $(X_0,g)$, $g=g_{X_0}$; one lets
\begin{equation*}
\tilde g=-dr^2+r^2
g=r^2\Big(-\frac{dr^2}{r^2}+g\Big)=\rho^2\Big(-\xXz^2\Big(\frac{d\rho}{\rho}+\frac{d\xXz}{\xXz}\Big)^2+\xXz^2g\Big),
\end{equation*}
which now gives, with $\xXz^2=-\mu$,
\begin{equation}\label{eq:extend-dS-to-Mink}
\tilde g=\rho^2\Big(\mu
\frac{d\rho^2}{\rho^2}+\frac{1}{2}\Big(\frac{d\rho}{\rho}\otimes
d\mu+d\mu\otimes \frac{d\rho}{\rho}\Big)-h(-\mu,\hat y,d\hat y)\Big),
\end{equation}
which is the same formula as \eqref{eq:extend-AH-to-Mink}, except the
appearance of $-\mu$ in the argument of $h$, corresponding to the
relationship between $g_{X_+}$ and $g_{X_0}$ when one started with
$g=g_{X_+}$, as expressed by \eqref{eq:induced-gXz}.

Thus, suppose we have an asymptotically de Sitter metric on a manifold
$(\overline{X_0},g_{X_0})$ with two boundary hypersurfaces $Y_\pm$ and a family of metrics
$\tilde h_\pm$ on $Y_\pm$ depending smoothly in an even fashion on the
boundary defining function $\xXz$ (i.e.\
smoothly on $\xXz^2$), and that $Y_\pm$ bound\footnote{If one starts
  with an $\overline{X_0}$ for which this is not the case, one can
  take two copies of it; the two copies of $Y_+$ bound now the
  manifold $Y_+\times[0,1]$ and similarly with $Y_-$.} manifolds with boundary
$\overline{X_\pm}$. Then one can put an asymptotically hyperbolic metric $g_\pm$ of the
form
$$
\frac{d\xXpm^2+h_\pm(-\xXpm^2,\hat y,d\hat y)}{\xXpm^2}
$$
near $Y_\pm=\pa X_\pm$ (relative to a chosen product decomposition, with a
factor $[0,\ep)_{\xXpm}$ corresponding to the boundary defining
function $\xXpm$) on
$\overline{X_\pm}$, and let $\mu=\xXpm^2$ on $\overline{X_\pm}$. Further,
we define a compact manifold
with boundary by
\begin{equation}\label{eq:tilde-X-const}
\tilde X=\overline{X_{+,\even}}\cup \overline{X_{0,\even}}\cup
\overline{X_{-,\even}},
\end{equation}
with the summands
smoothly identified at
the boundaries using the product decomposition used in transferring
the metric. Then we define a Lorentzian metric $\tilde g$ on
$\RR^+_\rho\times\tilde X$ by the respective form
\eqref{eq:extend-AH-to-Mink}-\eqref{eq:extend-dS-to-Mink} with $h$
understood as $\xXpm^2 g_\pm-d\xXpm^2$, resp.\ $-\xXz^2 g_0+d\xXz^2$
away from a neighborhood of $Y_\pm$; these definitions extend smoothly and consistently to
$\mu=0$ (i.e. $\RR^+\times Y_\pm$).

Returning to the previous discussion, when we started out with
$\overline{X_+}$, we can construct a global space $\tilde X$ by taking
two copies of $\overline{X_+}$, denoting the second copy by
$\overline{X_-}$, letting $Y_\pm=\pa\overline{X_\pm}$, and
$\overline{X_0}=Y_+\times[0,1]_s$, and defining $\tilde X$ as in
\eqref{eq:tilde-X-const}, with the corresponding identifications. This
defines asymptotically de Sitter metrics near the boundaries of
$\overline{X_0}$. Using the product structure on $\overline{X_0}$ this
can be extended to a Lorentzian metric on $X_0$ of a warped product
form $f(s)\,ds^2-h_0(s,\hat y,d\hat y)$ on $(0,1)_s\times Y_+$ with $f>0$, $h_0$
positive definite; note that this matches
the metric near $Y_\pm$ if $h_0$ is appropriately chosen, and all
null-geodesics indeed tend to $Y_\pm$ as the parameter along them
approaches infinity, so indeed this fits into the asymptotically de
Sitter framework described in the introduction.

Now the Mellin transform of $\Box_{\tilde g}$
gives rise to a smooth family of operators $\tilde P_\sigma$ on $\tilde X$, related to
$P_\sigma$ in \eqref{eq:AH-spect}-\eqref{eq:dS-spect} via the same procedure as in the Minkowski
setting. In summary, we have
shown:

\begin{prop}
Given an asymptotically hyperbolic space $(X_+,g_{X_+})$, resp.\ an asymptotically
de Sitter space $(X_0,g_{X_0})$, after possibly replacing
$(X_0,g_{X_0})$ by two copies of the same space, there is a `global' space $\tilde X$, of the form
\eqref{eq:tilde-X-const} with the not already given constituent pieces
asymptotically hyperbolic in case of $(X_\pm,g_{X_\pm})$ and
asymptotically de Sitter in case of $(X_0,g_{X_0})$, and there is an operator
$\tilde P_\sigma\in\Diff^2(\tilde X)$ on $\tilde X$, such that the restriction
of $\tilde P_\sigma$ to $X_\pm$, resp.\ $X_0$, is given by
\eqref{eq:P-rel-Xde}, with $P_\sigma$ as in \eqref{eq:AH-spect},
resp.\ \eqref{eq:dS-spect}.
\end{prop}

The requirements for the
analysis of $\tilde P_\sigma$ in \cite{Vasy-Dyatlov:Microlocal-Kerr}
involve the principal symbol globally as well as the imaginary part of
the subprincipal
symbol at $N^*Y_\pm$, with the latter entering since they determine the
threshold regularity at radial points. Further, if one wants to obtain high energy
estimates, letting $|\sigma|\to\infty$ in strips $|\im\sigma|<C$, one
also needs information on the principal symbol in the high
energy/large parameter sense. Here we do not address the latter (it
involves e.g.\ the non-trapping nature of the asymptotically
hyperbolic spaces), but mention that these are encoded in the
b-principal symbol of $\Box_{\tilde g}$ (which is the dual metric function), and indeed even the
$\sigma$-dependence of the subprincipal symbol can be read off from
the b-principal symbol of $\Box_{\tilde g}$.

The requirements on the principal symbol are satisfied in view of the
limiting behavior of the null-geodesics on the asymptotically de
Sitter space; apart from the behavior of the latter, the other
requirements were all checked in
\cite[Section~4]{Vasy-Dyatlov:Microlocal-Kerr} and
\cite[Section~3]{Vasy:Microlocal-AH}; the complex absorption added
there is not needed as we regard one of the radial sets $N^*Y_+$ and
$N^*Y_-$ as the region from which we start propagating estimates, the
other as the region
towards which we propagate estimates, as was done in the recent work \cite[Section~5]{Baskin-Vasy-Wunsch:Radiation}.
Thus, what is left is finding the subprincipal symbol at $N^*Y_\pm$,
and what is left in this
is finding a
$\sigma$-independent constant, which again, at most shifts by a
constant what function spaces should be used in the Fredholm
analysis. In turn, this constant can be found by
formal self-adjointness considerations as it is the principal symbol of
$\frac{1}{2\imath}(\tilde P_\sigma-\tilde P_\sigma^*)$ at the radial
set. The latter
vanishes for $\sigma$ real, as $\rho^2\rho^{(n-1)/2}\Box_{\tilde
  g}\rho^{-(n-1)/2}$ is formally self-adjoint with respect to the $\RR^+$-invariant
b-density $\rho^{-(n+1)}\,d\tilde g$, hence the Mellin transform is
formally self-adjoint for $\sigma$ real with respect to a density
$\omega$ on $\tilde X$ such that $\rho^{-(n+1)}\,d\tilde
g=\frac{d\rho}{\rho}\omega$ (cf.\ \cite[Section~3.3]{Vasy-Dyatlov:Microlocal-Kerr}).
It is actually instructive to compute this subprincipal symbol (rather
than just its imaginary part)
at $N^*Y$, $Y=Y_+\cup Y_-$, cf.\ \cite[Section~3]{Vasy:Analytic-forms} for the general setting of
differential forms;
one obtains that, with $\Vf_b(\tilde X;Y)$ denoting set of vector
fields on $\tilde X$ tangent to $Y$,
\begin{equation*}
\cM_\rho\rho^2\Box_{\tilde g}\cM_\rho^{-1}=(4\pa_\mu\mu\pa_\mu-4(\imath\tilde\sigma+(n-1)/2)\pa_\mu)+Q,
\ Q\in \Vf_b^2(\tilde X;Y),
\end{equation*} 
or
\begin{equation}\begin{aligned}\label{eq:tilde-P-sigma-form}
\tilde P_\sigma&=\cM_\rho\rho^2\rho^{(n-1)/2}\Box_{\tilde
  g}\rho^{-(n-1)/2}\cM_\rho^{-1}\\
&=(4\pa_\mu\mu\pa_\mu-4\imath\sigma\pa_\mu)+Q,
\ Q\in \Vf_b^2(\tilde X;Y).
\end{aligned}\end{equation}
This means $(\mu\pm \imath 0)^{\imath\sigma}$ are approximate elements
of the distributional kernel of $\tilde P_\sigma$ (in that they solve
$\tilde P_\sigma u=0$ modulo two orders better, namely smooth
multiples of $(\mu\pm \imath 0)^{\imath\sigma}$, than a priori
expected in view of the second order nature of $\tilde P_\sigma$: one
order of gain comes from $N^*Y$ being characteristic for the operator
and $(\mu\pm \imath 0)^{\imath\sigma}$ is conormal to this, but the
second order gain encodes the correct behavior of the subprincipal
symbol. Note that these distributions lie in $H^s$ for
$s<-\im\sigma+1/2$. Since in our global problem we are interested in
solutions of $\tilde P_\sigma u=f$ which are smooth at the future
light cone, $Y_+=\pa_+X_0$, if $f$ is smooth, we need to propagate
estimates from $Y_+=\pa_+X_0$ to $Y_-=\pa_- X_0$, and thus we need to use
Sobolev spaces which are stronger than the above threshold regularity,
$-\im\sigma+1/2$, at $Y_+=\pa_+X_0$, but are weaker than it at
$Y_-=\pa_-X_0$. Thus, as in
\cite[Section~5]{Baskin-Vasy-Wunsch:Radiation}, see also the Appendix
of that paper, we need variable order Sobolev spaces $H^s$, where $s$
is a $\CI$ function on $S^*\tilde X$ (though in this case one can take it to
be a function simply on $\tilde X$), corresponding to
$s_{\mathrm{past}}$ of \cite[Section~5]{Baskin-Vasy-Wunsch:Radiation},
so
\begin{enumerate}
\item
$s|_{N^*\pa_+X_0}>1/2-\im\sigma$, constant near $N^*\pa_+X_0$,
\item
$s|_{N^*\pa_-X_0}<1/2-\im\sigma$, constant near $N^*\pa_-X_0$,
\item
$s$ is monotone along the null-bicharacteristics (which all go from
$N^*\pa_+X_0$ to $N^*\pa_-X_0$ or vice versa).
\end{enumerate}
Then the spaces for Fredholm analysis are
\begin{equation}\label{eq:Pt-sigma-spaces}
\tilde P_\sigma:\cX^s\to\cY^{s-1},\ \cX^s=\{u\in H^s:\ \tilde P_\sigma u\in
H^{s-1}\},
\ \cY^{s-1}=H^{s-1},
\end{equation}
thus $\tPsb:\cY^{s-1}\to\cX^s$ is a meromorphic Fredholm family;
see \cite[Section~5]{Baskin-Vasy-Wunsch:Radiation} for details. Here
the subscript `$\past$' is added to denote the function spaces we are
using, which amounts to propagating regularity towards the past, i.e.\
$\pa_-X_0$: reversing the roles of $\pa_+X_0$ and $\pa_- X_0$ in the
definition of the function spaces would result in the the future
solution operator $\tPsf$.

\section{The global operator and the conformally compact spaces}\label{sec:relationship}

The
solution operator $\tPsb$ considered above now gives the solution operator
for the backward Cauchy problem for the spectral family of
$\Box_{X_0}$ as well as the resolvent for
$\Delta_{X_\pm}$.
This connection has been explored in
\cite{Vasy-Dyatlov:Microlocal-Kerr} and \cite{Vasy:Microlocal-AH} in
the asymptotically hyperbolic and de Sitter setting (the two setting
considered separately), and in \cite{Baskin-Vasy-Wunsch:Radiation} in
this generality (except that a compact $M$ was taken satisfying
various additional non-trapping conditions, but for the purposes of
the discussion here the latter are irrelevant). Here we expand this
discussion and include the Poisson operators and scattering operators
in it; the latter enter in perhaps surprising ways.

Sometimes
we write $\xXzpm$ for the boundary defining function
when we work near the future and past
boundaries $\pa_\pm X_0$ of de Sitter space to emphasize the local
nature of the expansion; these are understood to
be equal to $\xXz$ near the relevant boundary $\pa_\pm X_0$. {\em Further,
as the only smooth structure used below is the even one (corresponding
to the restriction of the smooth structure of $\tilde X$), below $\CI(\overline{\Xd})$ stands for
$\CI(\overline{\Xde})$, $\bullet=+,-,0$, unless otherwise noted.}

To elaborate on the connection mentioned above, concretely one
has, e.g.\ on $\CI_c(X_+)$, for $\im\sigma\gg 0$,
\begin{equation}\begin{aligned}\label{eq:global-to-hyp-inv}
&\cR_{X_+}(\sigma)=\Big(-\Delta_{X_+}+\sigma^2+\Big(\frac{n-1}{2}\Big)^2\Big)^{-1}\\
&\qquad=\xXp^{-\imath\sigma+(n-1)/2}\tPsb\xXp^{\imath\sigma-(n-1)/2 -2},\\
\end{aligned}\end{equation}
where the inverse on the left hand side is the inverse given by the
essential self-adjointness (on $\CI_c(X_+)$) and positivity of
$\Delta_{X_+}$. Notice that then the equality of the extreme left and
right hand sides holds for all $\sigma\in\Cx$ as the equality of
meromorphic families; alternatively, as in \cite{Vasy:Microlocal-AH}
the right hand side can be used to {\em define} the analytic
continuation of the resolvent of $\Delta_{X_+}$, i.e.\ $\cR_{X_+}(\sigma)$.
On the other hand,
on $\CI_c(X_0)$ the {\em backward}, or {\em past-oriented}, solution operator $\cR_{X_0,\past}(\sigma)$ is given by
\begin{equation}\begin{aligned}\label{eq:global-to-dS-inv}
&\cR_{X_0,\past}(\sigma)=\Big(\Box_{X_0}-\sigma^2-\Big(\frac{n-1}{2}\Big)^2\Big)^{-1}\\
&\qquad=\xXz^{-\imath\sigma+(n-1)/2}\tPsb\xXz^{\imath\sigma-(n-1)/2 -2}.
\end{aligned}\end{equation}

The former, \eqref{eq:global-to-hyp-inv}, was extensively discussed in
\cite{Vasy-Dyatlov:Microlocal-Kerr} and \cite{Vasy:Microlocal-AH}:
applied to $f\in\CI_c(X_+)$, both sides give an element of $L^2(X_+,dg_+)$ when
$\im\sigma\gg 0$ since $\tPsb$ maps into
$\CI(\overline{X_+})$, and in view of \eqref{eq:relate-Mellin-mod}
both sides satisfy that
$-\Delta_{X_+}+\sigma^2+\Big(\frac{n-1}{2}\Big)^2$ applied to them
yields $f$; since there is a unique element of $L^2(X_+,dg_+)$ with
this property, the claim follows.

To check the latter claim, \eqref{eq:global-to-dS-inv}, we first note that
\begin{equation}\label{eq:supp-map-P-1}
f\in\CI_c(X_0)\Rightarrow \supp \tPsb\xXz^{\imath\sigma-(n-1)/2 -2} f\cap
\overline{X_+}=\emptyset.
\end{equation}
We give two different arguments for
this. One is essentially a
direct application of Proposition~3.9 of
\cite{Vasy-Dyatlov:Microlocal-Kerr}. This proposition uses complex
absorption, but in a way that makes the proof go through without
changes in our setting: $Q_\sigma$ enters there only to make the
$P_\sigma$ into a Fredholm family, which we have here through control
of the global dynamics. The conclusion is that, using $-\mu$ as the
time function $\mathfrak{t}$ of \cite{Vasy-Dyatlov:Microlocal-Kerr}
near $\pa_+X_0$ (where it is time-like in $X_0$), $\tPsb$
propagates supports forward in $\mathfrak{t}$, i.e.\ backwards in
$\mu$, giving the desired conclusion. For an alternative proof of
\eqref{eq:supp-map-P-1} note that for
$f\in\CI_c(X_0)$, $\xXz^{\imath\sigma-(n-1)/2 -2}f$ vanishes in
$X_+$. Thus, $\tPsb\xXz^{\imath\sigma-(n-1)/2 -2} f$ also
vanishes there since this restriction is given by $\cR_{X_+}(\sigma)$
(the analytic continuation of the resolvent of $\Delta_{X_+}$, with
argument as in \eqref{eq:global-to-hyp-inv})
applied to the function $0$ by what we have shown.
But $\tPsb\xXz^{\imath\sigma-(n-1)/2 -2} f$ is $\CI$
near $\pa X_+=\pa_+ X_0$ (the future boundary of de Sitter space), and
thus the restriction to $\overline{X_0}$ vanishes to infinite order at
$\pa_+ X_0$, so the same remains true after multiplication by $\xXz^{-\imath\sigma+(n-1)/2}$. Calling the result $u$, which thus
satisfies $(\Box_{X_0}-\sigma^2-(\frac{n-1}{2})^2)u=f$, a slight modification of
\cite[Proposition~5.3]{Vasy:De-Sitter} gives that (for $f$ compactly
supported in $X_0$) $u$ vanishes identically near $\pa_+X_0$. The
slight modification we are referring to is that as stated,
\cite[Proposition~5.3]{Vasy:De-Sitter}
applies only for real
$\sigma$, but as the spectral variable is semiclassically
two orders below the principal term, it does not affect the Carleman
estimate argument presented there (it affects the error term $R_2$ in the proof by a
term in $h^2\Diff^0_{0,h}(X_0)$ with the notation of that paper, which
does not change the fact that $R_2$ is in the class stated
there). Note that the notion of semiclassicality is very different in
this Carleman estimate of
\cite{Vasy:De-Sitter} from that of \cite{Vasy-Dyatlov:Microlocal-Kerr}
since it is semiclassicality with respect to an exponential conjugation parameter, not $|\sigma|^{-1}$. Returning to $u$, this proves that $u$ is the backward
solution for the de Sitter Klein-Gordon equation.

To complete the picture, consider also when $f$ is supported in
$X_-$. To be clear we write $\mu_-$ for its boundary
defining function (which is positive in $X_-$), and we similarly write
$\xXm$, etc. Then by our argument thus far, $\tPsb\xXm^{\imath\sigma-(n-1)/2 -2} f$ vanishes
outside $\overline{X_-}$, i.e.\ is supported in
$\overline{X_-}$. Further, just under the assumption that
$f\in\CI(\tilde X)$
(i.e.\ without support assumptions),
$u=\tPsb f$ has $\WF(u)\subset N^*\pa X_-$, and indeed
has an expansion there, see \cite[Corollary~6.9]{Baskin-Vasy-Wunsch:Radiation},
namely if $\imath\sigma\notin\ZZ$ then
\begin{equation}\begin{aligned}\label{eq:expansion-at-outgoing-rad-set}
&u=\vXtmp+\vXtmm+\vXtmz,\\
&\vXtmpm=\aXtmpm (\mu_-\pm \imath 0)^{\imath\sigma},\qquad \aXtmpm,\vXtmz\in\CI(\tilde X).
\end{aligned}\end{equation}
Note that there is a sign switch in
\cite[Corollary~6.9]{Baskin-Vasy-Wunsch:Radiation} in $\sigma$
compared to the setting here; this is due to the use of a homogeneous
degree $1$ function in defining the Mellin transform here and its
reciprocal, i.e.\ a homogeneous degree $-1$ function (thus a defining
function of the boundary of the radial compactification of the space-time), being used in
\cite{Baskin-Vasy-Wunsch:Radiation} to perform the Mellin
transform. Also, if $\imath\sigma\in\ZZ$, logarithmic terms appear in
the expression corresponding to the fact that $(\mu_-\pm
\imath 0)^{\imath\sigma+k}$ is $\CI$ if $\imath\sigma+k$ is a non-negative
integer; this property of being $\CI$ shows up as an obstacle in the construction of
\cite{Baskin-Vasy-Wunsch:Radiation} for $k\geq 0$ integer, hence the
restriction $\imath\sigma\notin\ZZ$ here (though the general case can
also be treated).
Again with $\imath\sigma\notin\ZZ$, the first two terms can be rewritten
in terms of the distributions $(\mu_-)_\pm^{\imath\sigma}$, of which
$(\mu_-)_+^{\imath\sigma}$ is supported in $\overline{X_-}$.
Thus, for $f$ supported in $X_-$, the fact that $u$ is supported in
$\overline{X_-}$ implies, apart from integer coincidences,
that\footnote{Indeed, the $(\mu_-)_+^{\imath\sigma}$ term has the
  desired support property, so one is reduced to observing that the
  sum
of a $\CI$ multiple, say $\phi$, of $(\mu_-)_-^{\imath\sigma}$
  and a $\CI$ function, say $\psi$, is actually $\CI$ if it is supported in
  $\overline{X_-}$, and thus can be written as a multiple (with
  vanishing derivatives at $\pa X_-$) of
  $(\mu_-)_-^{\imath\sigma}$. Indeed, if the sum is so supported, the mismatch in the powers of
  the Taylor series of $\phi$ and $\psi$ at $\pa X_-$ due to $\imath\sigma$
  non-integral shows that both Taylor series vanish at $\pa X_-$, so
  the summands are in fact both $\CI$, and thus so is the sum, as desired.}
$u=b(\mu_-)_+^{\imath\sigma}$. Correspondingly, 
$\tilde u=\xXm^{-\imath\sigma+(n-1)/2} u|_{X_-}$ satisfies
$$
\Big(-\Delta_{X_-}+\sigma^2+\Big(\frac{n-1}{2}\Big)^2\Big)\tilde u=f,
$$
and
$$
\tilde u=\xXm^{\imath\sigma+(n-1)/2} \tilde a,\ \tilde a\in\CI(\overline{X_-}).
$$
Now, for $\im\sigma\ll 0$ this gives that
\begin{equation}\label{eq:global-to-X_--inv}
\cR_{X_-}(-\sigma) f=\xXm^{-\imath\sigma+(n-1)/2}\tPsb\xXm^{\imath\sigma-(n-1)/2 -2} f;
\end{equation}
this then holds in general in the sense of meromorphic Banach space
valued operators, even near $\imath\sigma\in\ZZ$. Notice that the
right hand side gives an independent way of analytically continuing
$\cR_{X_-}(-\sigma)$, similarly to how \eqref{eq:global-to-hyp-inv}
gives the analytic continuation of $\cR_{X_+}(\sigma)$ from
$\im\sigma\gg 0$. In summary, we have shown:

\begin{prop}\label{prop:global-to-pieces}(See \cite[Proposition~7.3]{Baskin-Vasy-Wunsch:Radiation}.)
For any $\sigma$ for which $\tilde P(\sigma)$ is invertible, the
resolvents $\cR_{X_+}(\sigma)$, $\cR_{X_-}(-\sigma)$ and the backward
solution operator $\cR_{X_0,\past}(\sigma)$ are determined by $\tilde
P(\sigma)$; in particular they are regular at these points.
\end{prop}

We want to have a converse result as well, namely that the poles of
$\tilde P(\sigma)$ are a subset of poles associated to operators on
$X_\pm$ and $X_0$ apart from possible issues when
$\imath\sigma\in\ZZ$.
In order to do this, it is useful to consider
solution operators for the homogeneous PDE, i.e.\ where non-trivial
boundary data are specified -- these are the so-called Poisson
operators. We recall that
$$
\pa_+X_0=\pa X_+,\ \pa_-X_0=\pa X_-.
$$

First, given $\aXtpmz\in\CI(\pa X_0)$ and $\imath\sigma\notin\ZZ$ one can easily write down
approximate solutions of the form
\begin{equation}\label{eq:tilde-form}
\vXtpm=\aXtpm (\mu\pm \imath 0)^{\imath\sigma},\ \aXtpm|_{\pa X_0}=\aXtpmz,\
\aXtpm\in\CI(\tilde X),
\end{equation}
i.e.\ such that
$$
\tilde P_\sigma \vXtpm\in\CI(\tilde X);
$$
see \cite[Lemma~6.4]{Baskin-Vasy-Wunsch:Radiation} for details (which in turn essentially
follows \cite{RBMSpec}); the Taylor series of $\aXtpm$ at $\pa X_0$
are determined by $\aXtpmz$. Note
that the Taylor series of
$\aXtpm$ at $\pa X_0$
is determined locally (in the strong sense that any Taylor coefficient depends only
on finitely many derivatives of $\aXtpmz$ evaluated at the same point), so in particular if $\aXtpmz|_{\pa_- X_0}=0$
then $\aXtpm$ vanishes to infinite order at $\pa_-X_0$.

Similarly, purely from the perspective of
$X_\pm$ and $X_0$, given $\aXpmpmz\in\CI(\pa X_\pm)$,
$\aXzpmz\in\CI(\pa X_0)$ one can construct
\begin{equation}\begin{aligned}\label{eq:hat-check-form}
&\vXpmpm=\aXpmpm \xXpm^{(n-1)/2\pm \imath\sigma},\qquad \aXpmpm|_{\pa
 X_\pm}=\aXpmpmz,\ \aXpmpm\in\CI(\overline{X_\pm}),\\
&\vXzpm=\aXzpm \xXz^{(n-1)/2\pm \imath\sigma},\qquad \aXzpm|_{\pa
 X_0}=\aXzpmz, \ \aXzpm\in\CI(\overline{X_0}),
\end{aligned}\end{equation}
with
\begin{equation*}\begin{aligned}
&\Big(-\Delta_{X_\pm}+\sigma^2+\Big(\frac{n-1}{2}\Big)^2\Big)\vXpmpm=\fXpmpm\in\dCI(\overline{X_\pm}),\\
&\Big(\Box_{X_0}-\sigma^2-\Big(\frac{n-1}{2}\Big)^2\Big)\vXzpm=\fXzpm\in\dCI(\overline{X_0}).
\end{aligned}\end{equation*}
Note the distinction: while on $\tilde X$ `trivial' or `residual'
functions are those in $\CI(\tilde X)$ (with no vanishing specified
anywhere), on $\Xd$ they are those in $\dCI(\Xd)$ (i.e.\ with infinite
order vanishing at the boundary).

We make the following observation:

\begin{lemma}\label{lemma:smooth-match}
Regarded as smooth functions on $\overline{X_+}$, resp.\
$\overline{X_0}$ (with the even structure, i.e.\ of $\mu=\xXp^2$, resp.\
$-\mu=\xXz^2$ rather than $\xXp$ and $\xXz$), at $\pa
X_+=\pa_+X_0$, if $\aXppmz=\aXzpmz$ then $\aXppm$ and $\aXzpm$ have
the matching Taylor series as functions in $\mu\geq 0$, resp.\
$\mu\leq 0$ (i.e.\ the even coefficients are the same, the odd coefficients have
opposite signs).
\end{lemma}

Note that $X_+$ can be replaced by $X_-$ in this lemma.

\begin{proof}
We consider $\aXpmz$ and $\aXzmz$. We notice that in view of the
(modified) conjugation relating $\tilde P_\sigma$ to
$-\Delta_{X_+}+\sigma^2+(n-1)^2/4$ on the one hand and
$\Box_{X_0}-\sigma^2-(n-1)^2/4$ on the other, these both solve $\tilde
P_\sigma|_{\overline{X_+}}\aXpmz=0$ and $\tilde
  P_\sigma|_{\overline{X_0}}\aXzmz=0$ in Taylor series at $\pa_+
  X_0=\pa X_+$. Since the form \eqref{eq:tilde-P-sigma-form} of
  $\tilde P_\sigma$ shows that the Taylor series of $\CI$ functions in the approximate nullspace
  (modulo functions vanishing to infinite order at $\pa X_+$) of
  $\tilde P_\sigma$ is determined\footnote{As $\mu^j \CI(\tilde X)$ is
    mapped to $\mu^{j-1}\CI(\tilde X)$ for $j\geq 1$ integer by
    $\tilde P_\sigma$, with $\tilde P_\sigma(\mu^j
    b)-j(j-\imath\sigma)b\mu^{j-1}\in\mu^j\CI(\tilde X)$, the claim
    follows by induction, noting that $j(j-\imath\sigma)$ cannot
    vanish when $j\geq 1$ is an integer as $\imath\sigma$ is not an
    integer.}
by the restriction to $\pa X_+$, the
  result follows. For $\aXpmz$ and $\aXzmz$ the result follows by
  considering $\tilde P_{-\sigma}$ in place of $\tilde P_{\sigma}$.
\end{proof}

We can now define the Poisson operators:

\begin{prop}(See \cite[Section~1]{Joshi-Sa-Barreto:Inverse} for an
  explicit statement, and also \cite{Mazzeo-Melrose:Meromorphic}.)
Suppose $\imath\sigma\notin\ZZ$, and $\sigma$ is not a pole of
$\cR_{X_+}$.
Given $\bXppz\in\CI(\pa X_+)$ there is a
solution $\uXp$ of
$$
\Big(-\Delta_{X_+}+\sigma^2+\Big(\frac{n-1}{2}\Big)^2\Big)\uXp=0
$$
with $\uXp=\vXpp+\vXpm$, $\vXppm$ of the form \eqref{eq:hat-check-form}, with
$\aXppz=\bXppz$.

Further, a solution $\uXp$ of this form is unique provided
$\imath\sigma\notin\ZZ$ and $\sigma^2+\Big(\frac{n-1}{2}\Big)^2$ is
not an $L^2$-eigenvalue of $\Delta_{X_+}$.
\end{prop}

\begin{rem}
Note that $\aXpmz$, i.e.\ the renormalized boundary value of $\vXpm$, is not specified.
\end{rem}

\begin{Def}\label{def:Poisson-X+}
The Poisson operator $\cPXp(\sigma):\CI(\pa
X_+)\to\dist(\overline{X_+})$
is defined as the meromorphic map $\bXppz\mapsto \uXp$ for
$\imath\sigma\notin\ZZ$.

The scattering matrix on $X_+$ is the operator
$\cSXp(\sigma):\CI(\pa X_+)\to\CI(\pa X_+)$ given by
$\cSXp(\sigma):\bXppz=\aXppz\mapsto \aXpmz$ with the notation of the
proposition and \eqref{eq:hat-check-form}.
\end{Def}

\begin{rem}
We could define $\cPXp^{-}(\sigma)$ similarly, in which $\aXpmz$
is specified in place of $\aXppz$, but this is just
$\cPXp(-\sigma)$ as reversing the sign of $\sigma$ interchanges
the two functions $\vXppm$. In particular, this gives $\cSXp(\sigma)=\cPXp(-\sigma)^{-1}\cPXp(\sigma)$.
\end{rem}

\begin{proof}
While this result is stated in \cite{Joshi-Sa-Barreto:Inverse}, we
give a summary of the argument.

For existence, $\uXp$ is given by
first constructing $\vXpp$ as above from $\aXppz$, and then for
$\sigma$ not a pole of $\cR_{X_+}$,
$$
\uXp=\vXpp-\cR_{X_+}(\sigma)\fXpp,
$$
with the second term of the form $\vXpm$ indeed.

Now consider uniqueness. The difference of two such $\uXp$ is of the
form $\vXpm$ necessarily since the leading coefficient $\aXppz$
determines the full Taylor series of $\aXpp$ (taking into account the
evenness of the Taylor series in terms of $\xXp$ to separate $\vXpp$
and $\vXpm$).
If $\im\sigma> 0$ and $\sigma^2+(n-1)^2/4$ is not
an $L^2$-eigenvalue of $\Delta_{X_+}$,
uniqueness follows
since $\vXpm$ is then in $H^2_0(\overline{X_+})$ (understood relative
to the non-even, i.e.\ standard, smooth structure). In general one can
show by a pairing argument, see \cite{Joshi-Sa-Barreto:Inverse}, which
in turn follows \cite{RBMSpec} that in fact the leading coefficient
$\aXpmz$ vanishes and then in fact $\vXpm$ is in
$\dCI(\overline{X_+})$, and then one can finish the argument as above.
\end{proof}

We can analogously
define a Poisson operator for $X_0$ at $\pa_+X_0$, but here we specify both
$\aXzpmz|_{\pa_+ X_0}$:

\begin{prop}(See \cite[Theorem~5.5]{Vasy:De-Sitter}.)
Suppose $\sigma\notin\imath\ZZ$. Given $\bXzpmz\in\CI(\pa_+ X_0)$ there is a
unique solution $\uXz$ of
$$
\Big(\Box_{X_0}-\sigma^2-\Big(\frac{n-1}{2}\Big)^2\Big)\uXz=0
$$
with $\uXz=\vXzp+\vXzm$, $\vXzpm$ of the form \eqref{eq:hat-check-form}, with
$\aXzpm|_{\pa_+ X_0}=\bXzpmz$.
\end{prop}

\begin{rem}
Note that there are two boundary hypersurfaces of $X_0$; we are
specifying both pieces of data $\aXzpmz$ at $\pa_+ X_0$ and neither of
them at $\pa_- X_0$.

Also, in \cite{Vasy:De-Sitter} only $\sigma^2$ real was considered,
but allowing general $\sigma\in\Cx$ causes only minimal changes to
the arguments. See also the remarks following \eqref{eq:supp-map-P-1}
in this regard.
\end{rem}

\begin{proof}
For existence, with $\vXzp,\vXzm$ as in \eqref{eq:hat-check-form} corresponding to
$\aXzpmz|_{\pa_+X_0}=\bXzpmz$ and $\aXzpmz|_{\pa_-X_0}=0$,
let
$$
\uXz=\vXzp+\vXzm-\cR_{X_0,\past}(\sigma)(\fXzp+\fXzm),
$$
with the inverse being the backward solution of the wave equation;
this has all the desired properties as shown in \cite{Vasy:De-Sitter}.
Uniqueness follows since the homogeneous PDE has no solutions which
vanish to infinite order at $\pa_+X_0$ as shown in \cite{Vasy:De-Sitter}.
\end{proof}

\begin{Def}\label{def:Poisson-X0}
The backward Poisson operator
$$
\cPXzb(\sigma):\CI(\pa_+X_0)\oplus
\CI(\pa_+X_0)\to\dist(\overline{X_0})
$$
is given by
$\cPXzb(\sigma) (\bXzpz,\bXzmz)=\uXz$ in the notation of
the proposition, while the scattering matrix
$$
\cSXtb(\sigma):\CI(\pa_+X_0)\oplus
\CI(\pa_+X_0)\to \CI(\pa_-X_0)\oplus
\CI(\pa_-X_0)
$$
is given by
$$
\cSXtb(\sigma)  (\bXzpz,\bXzmz)=(\aXzpz|_{\pa_-X_0},\aXzmz|_{\pa_-X_0}).
$$
\end{Def}

\begin{rem} Here
the index `$\past$' of $\cPXzb(\sigma)$ denotes that we are
solving the equation backwards, from $\pa_+X_0$ to $\pa_-X_0$. The
forward Poisson operator $\cPXzf(\sigma)$ is defined similarly, with
the data $(\aXzpz|_{\pa_-X_0},\aXzmz|_{\pa_-X_0})$ specified.

We also remark that replacing $\sigma$ by $-\sigma$ simply switches
the two pieces of data $\cPXzb(\sigma)$ is applied to, i.e.\ if $J$ is this
exchange operator then $\cPXzb(-\sigma)=\cPXzb(\sigma)J$. This is in
contrast to the asymptotically hyperbolic space, in which
$\cPXp(\sigma)$ and $\cPXp(-\sigma)$ are related by the much more
complicated scattering matrix $\cSXp(\sigma)$: $\cPXp(-\sigma)\cSXp(\sigma)=\cPXp(\sigma)$.
\end{rem}

Finally, we also have a Poisson
operator for the Mellin transformed global operator,
specifying both $\aXzpmz|_{\pa_+ X_0}$ again:

\begin{prop}
Suppose $\tilde P_\sigma$ is invertible as a map \eqref{eq:Pt-sigma-spaces}.
Then given $\bXtpmz\in\CI(\pa_+ X_0)$ there is a
unique solution $u$ of
$$
\tilde P_\sigma  u=0
$$
with
\begin{equation}\label{eq:Poisson-Xt-form}
\uXt=\vXtp+
\vXtm+\vXtz,
\end{equation}
with $\vXtpm$ of the form \eqref{eq:tilde-form}, with
$\aXtpmsz|_{\pa_+ X_0}=\bXtpmz$, and with $\vXtz\in\CI(\tilde X)$.
\end{prop}

\begin{proof}
Again, we let
$\vXtpm$ be as above with $\aXtpmsz|_{\pa_+ X_0}=\bXtpmz$,
$\aXtpmsz|_{\pa_- X_0}=0$ (so $\vXtpm$ is $\CI$ at $\pa_- X_0$), and then
$$
\uXt=\vXtp+
\vXtm-\tPsb(\fXtp+ \fXtm),
$$
is the unique distributional solution of $\tilde P_\sigma u=0$ with 
$\uXt-(\vXtp+\vXtm)$ having wave front set disjoint from $N^*\pa X_+$ (which
properties would hold for any $u$ of the desired form, thus giving uniqueness).
Further, $\uXt-(\vXtp+\vXtm)$ has wave front set in $N^*\pa X_-$, and indeed its structure given by
\eqref{eq:expansion-at-outgoing-rad-set} at $\pa X_-$, so $\uXt$ has the
decomposition claimed in the proposition.
\end{proof}

\begin{Def}\label{def:Poisson-Xt}
The backward Poisson operator
$$
\cPXtb(\sigma):\CI(\pa_+X_0)\oplus
\CI(\pa_+X_0)\to\dist(\tilde X)
$$
is given by
$\cPXtb(\sigma) (\bXtpz,\bXtmz)=\uXt$ in the notation of
the proposition, while the scattering matrix
$$
\cSXtb(\sigma):\CI(\pa_+X_0)\oplus
\CI(\pa_+X_0)\to \CI(\pa_-X_0)\oplus
\CI(\pa_-X_0)
$$
is given by
$$
\cSXtb(\sigma)  (\bXtpz,\bXtmz)=(\aXtpsz|_{\pa_-X_0},\aXtmsz|_{\pa_-X_0}).
$$
\end{Def}

We now work out the relationships between these operators. Thus, let
$$
\uXt=\cPXtb(\sigma)(\bXtpz,\bXtmz).
$$
Keeping in mind that $\mu=\xXp^2$, in view of \eqref{eq:Poisson-Xt-form},
$\uXp=\xXp^{-\imath\sigma+(n-1)/2} \uXt|_{X_+}$ satisfies
$$
\Big(-\Delta_{X_+}+\sigma^2+\Big(\frac{n-1}{2}\Big)^2\Big)\uXp=0,
$$
with $\uXp=\vXpp+\vXpm$,
\begin{equation*}\begin{aligned}
&\vXppm=\aXppm \xXp^{(n-1)/2\pm \imath\sigma},\ \aXppm\in\CI(\overline{X_+}),
\end{aligned}\end{equation*}
with
$$
\aXpp|_{\pa X_+}=\bXtpz+\bXtmz
$$
since the distribution $(\mu\pm \imath 0)^s$
restricted to $\mu>0$ is just the function $\mu^s$, and with
\begin{equation}\label{eq:a-from-X+}
\aXpm|_{\pa X_+}=\vXtz|_{\pa X_+}.
\end{equation}
Correspondingly,
\begin{equation}\label{eq:hyp-Poisson-forw-gen}
\Big(\xXp^{-\imath\sigma+(n-1)/2} \cPXtb(\sigma)(\bXtpz,\bXtmz)\Big)\Big|_{X_+}=\uXp=\cPXp(\sigma)(\bXtpz+\bXtmz).
\end{equation}
As an aside, this means that
\begin{equation}\label{eq:hyp-Poisson-forw}
\cPXp(\sigma)(\bXppz)=\xXp^{-\imath\sigma+(n-1)/2} (\cPXtb(\sigma)(\bXppz,0))|_{X_+},
\end{equation}
and one could equally well use $(0,\bXppz)$ as the data for
$\cPXtb(\sigma)$.
Returning to $\uXp=\vXpp+\vXpm$, we can now identify $\aXpm|_{\pa X_+}$ in terms of the scattering matrix on $X_+$:
\begin{equation}\label{eq:a-from-X+-S}
\aXpm|_{\pa X_+}=\cSXp(\sigma)(\aXpp|_{\pa X_+})=
\cSXp(\sigma) (\bXtpz+\bXtmz).
\end{equation}
Thus, switching to the de Sitter side,
with
$\uXt=\cPXtb(\sigma)(\bXtpz,\bXtmz)$ still, with
$\mu=-\xXz^2$ now,
$\uXz=\xXz^{-\imath\sigma+(n-1)/2} \uXt|_{X_0}$ satisfies
$$
\Big(\Box_{X_0}-\sigma^2-\Big(\frac{n-1}{2}\Big)^2\Big)\uXz=0,
$$
with $\uXz=\vXzp+\vXzm$,
\begin{equation*}\begin{aligned}
&\vXzpm=\aXzpm \xXz^{(n-1)/2\pm \imath\sigma},\ \aXzpm\in\CI(\overline{X_0}),
\end{aligned}\end{equation*}
with
$$
\aXzp|_{\pa
  X_+}=e^{-\pi\sigma}\bXtpz+e^{\pi\sigma}\bXtmz
$$
since the distribution $(\mu\pm \imath 0)^s$
restricted to $\mu<0$ is just the function $e^{\pm \imath\pi
  s}|\mu|^s=e^{\pm \imath\pi s}\xXz^{2s}$, and with
$$
\aXzm|_{\pa X_+}=\vXtz|_{\pa_+X_0}=\cSXp(\sigma)(\bXtpz+\bXtmz)
$$
in view of \eqref{eq:Poisson-Xt-form} for the first equality and
\eqref{eq:a-from-X+} and \eqref{eq:a-from-X+-S} for the second.
Correspondingly,
\begin{equation*}\begin{aligned}
&\xXz^{-\imath\sigma+(n-1)/2}\cPXtb(\sigma)(\bXtpz,\bXtmz)|_{X_0}\\
&\qquad=
\uXz=
\cPXzb(\sigma)(e^{-\pi\sigma}\bXtpz+e^{\pi\sigma}\bXtmz,\cSXp(\sigma)(\bXtpz+\bXtmz)).
\end{aligned}\end{equation*}
Thus,
\begin{equation}\label{eq:dS-Poisson}
\cPXzb(\sigma) (\bXzpz,\bXzmz)=\xXz^{-\imath\sigma+(n-1)/2}\cPXtb(\sigma)(\bXtpz,\bXtmz)|_{X_0}
\end{equation}
with
\begin{equation}\label{eq:dS-b-to-a}
\begin{bmatrix}
  \bXtpz\\ \bXtmz\end{bmatrix}=\begin{bmatrix}e^{-\pi\sigma}&e^{\pi\sigma}\\
1&1\end{bmatrix}^{-1}\begin{bmatrix}\bXzpz\\  \cSXp(\sigma)^{-1}\bXzmz\end{bmatrix},
\end{equation}
assuming $\cSXp(\sigma)$ is invertible and $\sigma\notin\imath\ZZ$
so that the matrix itself is invertible.

Finally we can turn to $X_-$. As
recalled above, near $\pa_-X_0=\pa X_-$,
\begin{equation*}\begin{aligned}
&\uXt=\vXtmp+\vXtmm+\vXtz,\\
&\vXtmpm=\aXtmpm (\mu_-\pm \imath 0)^{\imath\sigma},\qquad
\aXtmpm|_{\pa X_-}=\aXtmpmsz,\\
&\aXtmpm,\vXtmpm\in\CI(X_0\cup\overline{X_-}).
\end{aligned}\end{equation*}
Thus, $\uXz=\xXz^{-\imath\sigma+(n-1)/2} \uXt|_{X_0}$ has
asymptotic expansion at $\pa X_-$ given by
\begin{equation*}\begin{aligned}
&\uXz=\vXzmp+\vXzmm,\\
& \vXzmpm=(\xXzm)^{(n-1)/2\pm\imath\sigma}\aXzmpm,\qquad \aXzmpm\in\CI(\overline{X_0}),
\end{aligned}\end{equation*}
and
\begin{equation}\label{eq:X_0-minus-asymp-rel}
\aXzmp|_{\pa X_-}=e^{-\pi \sigma}
\aXtmp|_{\pa_- X_0}+e^{\pi\sigma}\aXtmm|_{\pa_- X_0},\ \aXzmm|_{\pa_- X_0}=\vXtz|_{\pa_- X_0}.
\end{equation}
Correspondingly,
\begin{equation}\label{eq:dS-sc-data}
\cSXzb(\sigma)\begin{bmatrix} \bXzpz \\ \bXzmz\end{bmatrix}
=\begin{bmatrix}e^{-\pi \sigma}
\aXtmp|_{\pa_- X_0}+e^{\pi\sigma}\aXtmm|_{\pa_- X_0}\\ \vXtz|_{\pa_- X_0}\end{bmatrix},
\end{equation}
with $\aXtpmz$ and $\bXzpmz$ related as in \eqref{eq:dS-Poisson}-\eqref{eq:dS-b-to-a}.

Now, in $X_-$ the resolvent is in the dual regime relative to that of
the $X_+$ problem (cf.\ the appearance of $-\sigma$ vs.\ $\sigma$ in
the argument of the resolvents in Proposition~\ref{prop:global-to-pieces}), namely
$\uXm=(\xXm)^{-\imath\sigma+(n-1)/2} \uXt|_{X_-}$ solves
$$
\Big(-\Delta_{X_-}+\sigma^2+\Big(\frac{n-1}{2}\Big)^2\Big)\uXm=0,
$$
with asymptotics
\begin{equation*}\begin{aligned}
&\uXm=\vXmp+\vXmm,\\
&\vXmpm=(\xXm)^{(n-1)/2\pm\imath\sigma}\aXmpm,\qquad \aXmpm\in\CI(\overline{X_-}),
\end{aligned}\end{equation*}
and
$$
\aXmp|_{\pa X_-}=\aXtmp|_{\pa_- X_0}+\aXtmm|_{\pa_- X_0},\qquad \aXmm|_{\pa_- X_0}=\vXtz|_{\pa_- X_0}.
$$
Thus, much as in the case of the resolvent considered first above,
except using $\cPXm(-\sigma)$, so the coefficient of
$\xXm^{(n-1)/2-\imath\sigma}$, namely $v^{-}_0|_{\pa X_-}$, is the input,
\begin{equation}\begin{aligned}\label{eq:hyp-Poisson-back}
\cPXm(-\sigma)(v^{-}_0|_{\pa X_-})&=\uXm=(\xXm)^{-\imath\sigma+(n-1)/2} \uXt|_{X_-}\\
&=(\xXm)^{-\imath\sigma+(n-1)/2} \cPXtb(\sigma)(\aXtpsz,\aXtmsz)|_{X_-},
\end{aligned}\end{equation}
and
$$
\cSXm(-\sigma)\vXzm|_{\pa X_-}=\aXtmp|_{\pa_- X_0}+\aXtmm|_{\pa_- X_0}.
$$
Thus, using \eqref{eq:X_0-minus-asymp-rel},
$$
\begin{bmatrix}\Id&0\\0&\cSXm(-\sigma)\end{bmatrix}
\cSXzb(\sigma)\begin{bmatrix}
    \bXzpz\\ \bXzmz\end{bmatrix}=\begin{bmatrix}
e^{-\pi
  \sigma}&e^{\pi\sigma}\\1&1\end{bmatrix}\begin{bmatrix}\aXtmp|_{\pa_-
  X_0}\\ \aXtmm|_{\pa_- X_0}\end{bmatrix}.
$$
Combining this with \eqref{eq:dS-b-to-a} we have shown

\begin{thm}\label{thm:S-matrix}
For $\sigma\notin\imath\ZZ$, if $\sigma$ is not a pole of $\tPsb$
then the global Poisson operator $\cPXtb(\sigma)$ on $\tilde X$ determines
those of $X_\pm$ and $X_0$, $\cPXp(\sigma)$,
$\cPXm(-\sigma)$
and $\cPXzb(\sigma)$, and conversely, $\cPXp(\sigma)$,
$\cPXm(-\sigma)$
and $\cPXzb(\sigma)$
determine the
global Poisson operator $\cPXtb(\sigma)$.

Furthermore, for $\sigma$ as above,
\begin{equation*}\begin{aligned}
&\cSXtb(\sigma)(\bXtpz,\bXtmz)=\begin{bmatrix}\aXtmp|_{\pa_- X_0}\\ \aXtmm|_{\pa_-
   X_0}\end{bmatrix}\\
&\quad=\begin{bmatrix}
e^{-\pi \sigma}&e^{\pi\sigma}\\1&1\end{bmatrix}^{-1}
\begin{bmatrix}\Id&0\\0&\cSXm(-\sigma)\end{bmatrix}
\cSXzb(\sigma)
\begin{bmatrix}\Id&0\\0&\cSXp(\sigma)\end{bmatrix}\begin{bmatrix}e^{-\pi\sigma}&e^{\pi\sigma}\\
1&1\end{bmatrix}\begin{bmatrix}
 \bXtpz\\ \bXtmz\end{bmatrix},
\end{aligned}\end{equation*}
i.e.\ $\cSXtb(\sigma)$ is essentially the product of
$\cSXpm(\pm\sigma)$ and $\cSXzb(\sigma)$.
\end{thm}

\begin{proof}
First
\eqref{eq:hyp-Poisson-forw} shows that the global Poisson operator $\cPXtb(\sigma)$ 
determines $\cPXp(\sigma)$, and in particular $\cSXp(\sigma)$.
Next, \eqref{eq:dS-Poisson} that $\cPXtb(\sigma)$ determines
$\cPXzb(\sigma)$, and in particular $\cSXzb(\sigma)$.
Finally, \eqref{eq:hyp-Poisson-back} combined with
\eqref{eq:dS-sc-data}
show that $\cPXtb(\sigma)$ determines $\cPXm(\sigma)$.

For the converse, \eqref{eq:hyp-Poisson-forw-gen} shows that
$\cPXp(\sigma)$ determines the restriction of $\cPXtb(\sigma)$ to
$X_+$, and in particular the Cauchy data at future infinity,
$(\bXzpz,\bXzmz)$, for the de Sitter problem. Then
\eqref{eq:dS-Poisson} shows that the restriction of $\cPXtb(\sigma)$ to
$X_0$ is determined, and in particular the data for
$\cPXm(-\sigma)$ is determined. Then \eqref{eq:hyp-Poisson-back}
shows that the restriction of $\cPXtb(\sigma)$ to $X_-$ is
determined. Since $\imath\sigma$ is not a negative integer, the form
\eqref{eq:expansion-at-outgoing-rad-set} shows that these restrictions
determine $\cPXtb(\sigma)$ since there cannot be solutions of the
homogeneous equation supported at $\pa X_{\pm}$.
\end{proof}

Now, $\cSXpm(\sigma)$ are elliptic pseudodifferential
operators of (complex)
order $-2\imath\sigma$,
as shown by Joshi and S\'a Barreto\footnote{Note that Joshi and
S\'a Barreto use the spectral parameter $-\zeta(n-1-\zeta)$, with our
notation for the dimension of $X$, with $\re\zeta>(n-1)/2$ being the
physical half plane, corresponding to our $\sigma^2+(n-1)^2/4$ with
$\im\sigma>0$ being the physical half plane, so
$\sigma=\imath(\zeta-(n-1)/2)$ is the conversion between the two parameterizations.},
 \cite{Joshi-Sa-Barreto:Inverse},
so $\cSXm(-\sigma)$ has order $2\imath\sigma$. In particular, if
$\Delta_{\pa X_\pm}$ is the Laplacian of a metric on $\pa X_\pm$, say
of (a representative of the conformal class of) the conformal metric
$h$, then
\begin{equation}\label{eq:renorm-ah-sc}
(\Delta'_{\pa X_+})^{\imath\sigma}\cSXp(\sigma),\
\cSXm(-\sigma)(\Delta'_{\pa X_-})^{-\imath\sigma}
\end{equation}
are pseudodifferential operators of order $0$, where $\Delta'_{\pa
  X_+}$ is the operator that is $\Delta_{\pa
  X_+}$ on the orthocomplement of the nullspace of $\Delta_{\pa
  X_+}$ and the identity on the
  nullspace.\footnote{Other second order positive elliptic operators,
    bounded below by a positive constant,  would do equally
  well; with the choice of $\Delta'_{\pa X_+}$, the principal symbol of the $0$th order
  operators in \eqref{eq:renorm-ah-sc} is a constant $c_\sigma$,
  resp.\ $c_{-\sigma}$, dependent on
  $\sigma$ only via powers of $2$ and the $\Gamma$-function, see
  \cite[Theorem~1.1]{Joshi-Sa-Barreto:Inverse} and with
  $c_{\sigma}c_{-\sigma}=1$. \label{footnote:scat-norm}
}
Further, $\cSXzb(\sigma)$ is an elliptic Fourier
integral operator associated to the backward null-geodesic flow from
$\pa_+X_0$ to $\pa_-X_0$ as
shown by the author\footnote{Note that in \cite{Vasy:De-Sitter} the
  two summands are interchanged: the $x^{s_+(\lambda)}\wXzp$ term is
  put first, $x^{s_-(\lambda)}\wXzm$ is put second,
  $\wXzp,\wXzm\in\CI(\overline{X_0})$, which is the
  reverse of Definition~\ref{def:Poisson-X0} in view of
  \eqref{eq:sigma-to-s}. Further, the assumption in
  \cite{Vasy:De-Sitter} in the stated version of Theorem~1.2 is that
  $2\imath\sigma$ is not an integer, but as is explained below the
  statement of this theorem, if the metric is even, as in our case, $\imath\sigma$ not
an integer suffices.} in \cite{Vasy:De-Sitter}, with the property that
\begin{equation*}\begin{aligned}
\big((\Delta'_{\pa_-X_0})^{-s_-(\lambda)/2+n/4}\oplus&(\Delta'_{\pa_-X_0})^{-s_+(\lambda)/2+n/4}\big)\cSXzb(\sigma)\\
&\big((\Delta'_{\pa_+
  X_0})^{s_-(\lambda)/2-n/4}\oplus(\Delta'_{\pa_+ X_0})^{s_+(\lambda)/2-n/4}\big)
\end{aligned}\end{equation*}
is a Fourier integral operator of order $0$, where the spectral
parameter is $\lambda=\sigma^2+(n-1)^2/4$, and
$$
s_\pm(\lambda)=\frac{n-1}{2}\pm\sqrt{(n-1)^2/4-\lambda},
$$
with the
square root being the standard one in $\Cx\setminus(-\infty,0]$, which
means that
\begin{equation}\label{eq:sigma-to-s}
s_\pm(\lambda)=\frac{n-1}{2}\mp\imath\sigma,
\end{equation}
$\im\sigma>0$ being the physical half plane. Composing with
$$
(\Delta'_{\pa_-X_0})^{s_-(\lambda)/2-n/4}\oplus
(\Delta'_{\pa_-X_0})^{s_-(\lambda)/2-n/4}
$$
from the left and
$$
(\Delta'_{\pa_+X_0})^{-s_-(\lambda)/2+n/4}\oplus
(\Delta'_{\pa_+X_0})^{-s_-(\lambda)/2+n/4}
$$
from the right,
one still has an order $0$ Fourier integral operator, i.e.
\begin{equation}\label{eq:renorm-zero-sc}
\big(\Id\oplus(\Delta'_{\pa_-X_0})^{s_-(\lambda)/2-s_+(\lambda)/2}\big)\cSXzb(\sigma)\big(\Id\oplus(\Delta'_{\pa_+ X_0})^{s_+(\lambda)/2-s_-(\lambda)/2}\big)
\end{equation}
is $0$th order. Noting that
$(s_+(\lambda)-s_-(\lambda))/2=-\imath\sigma$,
\begin{equation*}\begin{aligned}
&\begin{bmatrix}\Id&0\\0&\cSXm(-\sigma)\end{bmatrix}
\cSXzb(\sigma)
\begin{bmatrix}\Id&0\\0&\cSXp(\sigma)\end{bmatrix}\\
&\qquad=\begin{bmatrix}\Id&0\\0&\cSXm(-\sigma)(\Delta'_{\pa
    X_-})^{-\imath\sigma}\end{bmatrix}
\Big(\big(\Id\oplus(\Delta'_{\pa_-X_0})^{\imath\sigma}\big)\cSXzb(\sigma)\big(\Id\oplus(\Delta'_{\pa_+ X_0})^{-\imath\sigma}\big)\Big)\\
&\hspace{3in}\begin{bmatrix}\Id&0\\0&(\Delta'_{\pa
    X_+})^{\imath\sigma}\cSXp(\sigma)\end{bmatrix},
\end{aligned}\end{equation*}
it
follows immediately that $\cSXtb(\sigma)$ is a Fourier integral
operator associated to the same flow, with principal symbol the same
as that of $\cSXzb(\sigma)$ in view of Footnote~\ref{footnote:scat-norm}.

\begin{cor}\label{cor:FIO}
For $\sigma\in\Cx$ with $\imath\sigma\notin\ZZ$, and $\sigma$ not a
pole of $\tPsb$, $\cSXtb(\sigma)$ is an elliptic $0$th order
Fourier integral operator associated with the null-geodesic flow from
$\pa_+X_0$ to $\pa_-X_0$ on $X_0$, with principal symbol the same as
that of the renormalized backwards scattering operator on $X_0$ as in
\eqref{eq:renorm-zero-sc} conjugated by the matrix
$$
\begin{bmatrix}e^{-\pi\sigma}&e^{\pi\sigma}\\
1&1\end{bmatrix}
$$
as in Theorem~\ref{thm:S-matrix}.
\end{cor}

We can now put together the {\em local} relationship between the resolvents of the
problems on $X_0$ and $X_\pm$ on the one hand, and
on $\tilde X$ on the other, namely the ingredients
\eqref{eq:global-to-hyp-inv}, \eqref{eq:global-to-dS-inv} and
\eqref{eq:global-to-X_--inv} of
Proposition~\ref{prop:global-to-pieces}, together with the global
understanding of the Poisson operators to show that not only does
$\tPsb$ determine the local inverses, but the converse
also holds. We remark that this has been partially explored in
\cite[Section~7]{Baskin-Vasy-Wunsch:Radiation}, in which the diagonal
elements of the matrix described in Theorem~\ref{thm:parts-to-whole}
were obtained, following \cite{Vasy-Dyatlov:Microlocal-Kerr}, in a
somewhat weaker sense (in terms of support properties of $f$ to which
$\tPsb$ is being applied).

Thus, given $f\in\CI(\tilde X)$, we first define a
distribution $\uXt$ (which in fact will be $\CI$ away from $\pa X_-$)
by defining its restrictions $\uXtXp$, $\uXtXz$, resp.\ $\uXtXm$
to $X_+$, $X_0$ resp.\ $X_-$, checking that $\uXtXp$ and $\uXtXz$ extend
smoothly to $\pa X_+$, hence $\uXt$ can defined to be smooth across $\pa
X_+$, and then analyzing the precise singularity of $\uXtXz$ and $\uXtXm$ at
$\pa X_-$ and using this to actually define a distribution near $\pa
X_-$ as well.

So first let
\begin{equation}\label{eq:hyp-to-global-inv}
\uXtXp=\xXp^{\imath\sigma-(n-1)/2}\cR_{X_+}(\sigma)\xXp^{-\imath\sigma+(n-1)/2
  +2}f|_{X_+}.
\end{equation}
Then $\uXtXp\in \CI(\overline{X_+})$ (in the even sense!) by the
mapping properties of the resolvent on $X_+$; let $\vXtXpmz=\uXtXp|_{\pa
  X_+}$. Next, we define $\uXtXz\in\CI(X_0)$ by
\begin{equation}\begin{aligned}\label{eq:dS-to-global-inv}
\uXtXz=&\xXz^{\imath\sigma-(n-1)/2}\cPXzb(\sigma)(0,\vXtXpmz)\\
&\qquad+
\xXz^{\imath\sigma-(n-1)/2}\cR_{X_0,\past}(\sigma)
\xXz^{-\imath\sigma+(n-1)/2 +2}f|_{X_0}.
\end{aligned}\end{equation}
Then $\uXtXz$ is $\CI$ up to $\pa_+ X_0$, and it
has an asymptotic expansion at $\pa_- X_0$ of the form
$$
\uXtXz=\vXtXzp+\vXtXzm,\ \text{with}\ \vXtXzp=(\xXzm)^{2\imath\sigma}\aXtXzp,\ \vXtXzm=\aXtXzm,
$$
with $\aXtXzpm$ being $\CI$ up to $\pa_- X_0$.
Here, $\uXtXp$ and $\uXtXz$ not only have the same restriction at $\pa
X_+$ (which is automatic by the definition of the Poisson operator),
but have matching Taylor series (in terms of the `even' smooth structure,
i.e.\ that of $\tilde X$) by Lemma~\ref{lemma:smooth-match}.
We next let
\begin{equation}\label{eq:X_--to-global-inv}
\uXtXm=\xXm^{\imath\sigma-(n-1)/2}\cPXm(-\sigma) \aXtXzm|_{\pa_- X_0}+\xXm^{\imath\sigma-(n-1)/2} \cR_{X_-}(-\sigma) \xXm^{-\imath\sigma+(n-1)/2 +2} f|_{X_-}.
\end{equation}
Then $\uXtXm$ has an asymptotic expansion at $\pa X_-$ of the form
$$
\vXtXmp+\vXtXmm,\ \vXtXmp=
\xXm^{2\imath\sigma}\aXtXmp,\ \vXtXmm=\aXtXmm,
$$
and $\aXtXmpm$ are $\CI$ up to $\pa X_-=\pa_-X_0$. Further, again, $\aXtXmm$ and $\aXtXzm$ not only have the same restriction at $\pa
X_-$ (which is automatic by the definition of the Poisson operator),
but have matching Taylor series by Lemma~\ref{lemma:smooth-match}.
Now notice that for $\sigma\notin\imath\ZZ$
there is a unique distribution defined near $\pa X_-$,
of the form
\begin{equation}\label{eq:pmi0-comb}
\aXtmp(\mu+\imath 0)^{\imath\sigma}+\aXtmm(\mu-\imath
0)^{\imath\sigma},
\end{equation}
$\aXtmpm$ being $\CI$ near $\pa X_-$,
whose restriction to $X_0$, resp.\ $X_-$ is $\vXtXzp$, resp.\
$\vXtXmp$. Indeed, the difference of any two such distributions
would be a differentiated delta distribution supported on $\pa X_-$,
which are never of this form if $\sigma\notin\imath\ZZ$, showing
uniqueness, while expanding $\aXtXzp$, $\aXtXmp$ and the
putative $\aXtpm$ in
Taylor series around $\pa X_-$, one is reduced to observing that one
must have for the $j$th term in the ($\mu$-based, i.e.\ even in terms
of $\xXd$) Taylor series
$$
\begin{bmatrix}\aXtXzmj\\ \aXtXmmj\end{bmatrix}=\begin{bmatrix}e^{-\pi(\sigma-\imath
    j)}&e^{\pi(\sigma-\imath
    j)}\\1&1\end{bmatrix}\begin{bmatrix}\aXtmpj\\ \aXtmmj\end{bmatrix},
$$
and in case $\sigma\notin\imath\ZZ$, the matrix on the right hand side
is invertible. Thus, there is a unique distribution $\uXt$ on $\tilde X$ which
is $\CI$ away from $\pa X_-$, which is of the
form
\begin{equation}\label{eq:form-at-pa-X_-}
\aXtmz+\aXtmp(\mu+\imath 0)^{\imath\sigma}+\aXtmm(\mu-\imath
0)^{\imath\sigma},
\end{equation}
with $\aXtmz, \aXtmpm$ being $\CI$ near $\pa X_-$, and whose
restrictions to $X_+$, resp.\ $X_0$, resp.\ $X_-$ are $\uXtXp$, resp,\
$\uXtXz$, resp.\ $\uXtXm$. This distribution satisfies $\tilde P_\sigma \uXt=f$ on
each of $X_+$, $X_0$ and $X_-$. Further, $\uXt$ being $\CI$ near
$\pa X_+$, $\tilde P_\sigma \uXt-f$ is $\CI$ there, vanishing on $X_0\cup X_+$,
thus vanishing near $\pa X_+$ as well, i.e.\ $\tilde P_\sigma \uXt-f$ is
supported at $\pa X_-$. But there $\uXt$ has the form
\eqref{eq:form-at-pa-X_-}, and thus $\tilde P_\sigma \uXt$ necessarily has a
similar form with the exponents decreased by $1$ (since $\tilde P_\sigma$ is
second order, but is characteristic on $N^*\pa X_-$). Correspondingly,
as long as $\sigma\notin\imath\ZZ$, $\tilde P_\sigma \uXt-f$ cannot be a sum of
differentiated delta distributions on $\pa X_-$, so the vanishing of
$\tilde P_\sigma \uXt-f$ away from $\pa X_-$ shows that $\tilde P_\sigma \uXt=f$.
Thus, given $\sigma\notin\imath\ZZ$ which is not a pole of $\cR_{X_\pm}(\pm\sigma)$, and
given $f\in\CI(\tilde X)$, we showed that $f=\tilde P_\sigma \uXt$.

\begin{prop}\label{prop:parts-to-whole-res}
For $\sigma\notin\imath\ZZ$, if $\sigma$ is not a pole of
$\cR_{X_\pm}(\pm\cdot)$, then $\sigma$ is not a pole of $\tPssb$.
\end{prop}

Combining Propositions~\ref{prop:global-to-pieces} and
\ref{prop:parts-to-whole-res} yields

\begin{thm}\label{thm:parts-to-whole}(Strengthened version of \cite[Proposition~7.3]{Baskin-Vasy-Wunsch:Radiation}.)
The poles of $\tPsb$ in $\Cx\setminus\imath\ZZ$ are
exactly the union of the poles of $\cR_{X_+}(\sigma)$ and
$\cR_{X_-}(-\sigma)$.

Furthermore, with the blocks $X_+$, $X_0$ and $X_-$ listed
left-to-right and top-to-bottom, and $(.)_{jk}$ referring to the $jk$
entry of this matrix to shorten the notation, and with
$\cPXzf(\sigma)^{-1}_j$ denoting the $j$th component of
$\cPXzf(\sigma)^{-1}$ ($j=1,2$, so $j=1$ corresponds to the
superscript $+$, $j=2$ to the superscript $-$ in Definition~\ref{def:Poisson-X0}), the matrix of $\tPsb$ is, column by column, (so $X_+$ is the first column, etc.)
\begin{equation*}\begin{aligned}
&(\tPsb)_{.1}=\begin{bmatrix}\xXp^{\imath\sigma-(n-1)/2}\cR_{X_+}(\sigma) \xXp^{-\imath\sigma+(n-1)/2+2}\\
\xXz^{\imath\sigma-(n-1)/2}\cPXzb(\sigma)
(0,\cPXp^{-1}(-\sigma)\cR_{X_+}(\sigma)
\xXp^{-\imath\sigma+(n-1)/2+2})\\ \xXm^{\imath\sigma-(n-1)/2}\cPXm(-\sigma) \cPXzf(\sigma)^{-1}_2\xXz^{-\imath\sigma+(n-1)/2} ()_{21}
\end{bmatrix},\\
&(\tPsb)_{.2}=\begin{bmatrix}0\\
\xXz^{\imath\sigma-(n-1)/2}\cR_{X_0}(\sigma)\xXz^{-\imath\sigma+(n-1)/2 +2}\\
\xXm^{\imath\sigma-(n-1)/2}\cPXm(-\sigma)
\cPXzf(\sigma)^{-1}_2\xXz^{-\imath\sigma+(n-1)/2} ()_{22}&\end{bmatrix},\\
&(\tPsb)_{.3}=\begin{bmatrix}0\\
0\\ \xXm^{\imath\sigma-(n-1)/2} \cR_{X_-}(-\sigma)
\xXm^{-\imath\sigma+(n-1)/2 +2}\end{bmatrix}.
\end{aligned}\end{equation*}
\end{thm}

\begin{rem}
We finally remark that excluding $\imath\sigma\in\ZZ$ in
\eqref{eq:pmi0-comb} was excessive; it suffices to rule out that
$\imath\sigma$ is a negative integer if we work in terms of the
distributions $\mu^{\imath\sigma}_\pm$ instead, i.e.\ $\im\sigma<1$
suffices there. Further, for $\im\sigma>-1$, all operators in
the two-by-two upper left block are well-defined (and holomorphic) even if
$\imath\sigma$ is an integer as long as $\sigma$ is not a pole of
$\cR_{X_+}(\sigma)$. Indeed, $\cPXp^{-1}(-\sigma)$ reads off the leading
asymptotic term of $\cR_{X_+}(\sigma)$, while, for $\im\sigma>0$, $\cPXzb(\sigma)$ solves
the de Sitter Klein-Gordon equation where the second, more decaying
(here we use $\im\sigma>0$) datum is specified, which makes sense in a
holomorphic manner even
in the case of integer $\imath\sigma$, and if we merely assume
$\im\sigma>-1$, the same conclusion holds though the specified
behavior, $x^{(n-1)/2-\imath\sigma}\aXzm|_{\pa_+X_0}$, is now possibly
the less
decaying one. (At $\im\sigma=-1$, constructing $\vXzm$ near $\pa_+X_0$
introduces logarithmic terms and changes the construction
significantly. This is still possible, as was done in
\cite{Vasy:De-Sitter}, but this seriously affects holomorphic arguments.) Thus, when composed with
restriction to $\overline{X_+}\cup X_0$ from the left and extension of
compactly supported functions on $\overline{X_+}\cup X_0$ from the
right, the only poles of $\tPsb$ are those of
$\cR_{X_+}(\sigma)$ and possibly $\sigma$ with $\imath\sigma$ an
integer with $\im\sigma\leq -1$. We also refer to
\cite[Remark~4.6]{Vasy-Dyatlov:Microlocal-Kerr}, where the same
conclusion is established via a different argument.
\end{rem}

\def\cprime{$'$} \def\cprime{$'$}

\end{document}